\documentclass{amsart}
\usepackage{dsfont, amssymb, mathtools}
\usepackage{calc,changepage}
\usepackage{csquotes}
\usepackage{hyperref}
\usepackage{enumitem}
\setlist{  
  parsep=0pt,
}
\usepackage{xcolor}
\definecolor{wb}{RGB}{51,153,255}
\usepackage{tikz-cd}

\usepackage[parfill]{parskip}
\numberwithin{equation}{subsection}
\usepackage{amsrefs}

\newcommand{\defeq}{\vcentcolon=}

\makeatletter
\def\moverlay{\mathpalette\mov@rlay}
\def\mov@rlay#1#2{\leavevmode\vtop{%
   \baselineskip\z@skip \lineskiplimit-\maxdimen
   \ialign{\hfil$\m@th#1##$\hfil\cr#2\crcr}}}
\newcommand{\charfusion}[3][\mathord]{
    #1{\ifx#1\mathop\vphantom{#2}\fi
        \mathpalette\mov@rlay{#2\cr#3}
      }
    \ifx#1\mathop\expandafter\displaylimits\fi}
\makeatother

\newcommand{\cupdot}{\charfusion[\mathbin]{\cup}{\cdot}}
\newcommand{\bigcupdot}{\charfusion[\mathop]{\bigcup}{\cdot}}

\newtheoremstyle{definitions}
 	{\topsep}
	{\topsep}
	{}
	{}
	{\bfseries}
	{:}
	{.5em}
	{}
  
\newtheoremstyle{lemmata}
	{\topsep}
	{\topsep}
	{\itshape} 
	{}
	{\bfseries}
	{:}
	{.5em}
	{}

\theoremstyle{lemmata}
\newtheorem{Theorem}[subsection]{Theorem}
\newtheorem{Lemma}[subsection]{Lemma}
\newtheorem{Corollary}[subsection]{Corollary}
\newtheorem{Proposition}[subsection]{Proposition}
\newtheorem{manualtheoreminner}{Theorem}
\newenvironment{manualtheorem}[1]{%
  \renewcommand\themanualtheoreminner{#1}%
  \manualtheoreminner
}{\endmanualtheoreminner}

\theoremstyle{definitions}
\newtheorem{Definition}[subsection]{Definition}
\newtheorem{Remark}[subsection]{Remarks}

\title{Invertible functions on non-archimedean symmetric spaces}
\author{Ernst-Ulrich Gekeler}
\date{\today}

\begin{document}

\maketitle

\small{FR 6.1 Mathematik, Universität des Saarlandes, Postfach 15 11 50 \\
D-66041 Saarbrücken \\
mail: \href{mailto:gekeler@math.uni-sb.de}{gekeler@math.uni-sb.de} \\
phone: ++49 681 302 2494 \\
ORCID: 0000-0002-2299-6705 \\
MSC Primary 32P05; Secondary 32C30, 32C36, 11F23, 11F85}

\begin{abstract}
	Let $u$ be a nowhere vanishing holomorphic function on the Drinfeld space $\Omega^{r}$ of dimension $r-1$, where $r \geq 2$. The logarithm $\log_{q}\lvert u \rvert$ of its absolute value may be 
	regarded as an affine function on the attached Bruhat-Tits building $\mathcal{BT}^{r}$. Generalizing a construction of van der Put in case $r=2$, we relate the group 
	$\mathcal{O}(\Omega^{r})^{*}$ of such $u$ with the group $\mathbf{H}(\mathcal{BT}^{r}, \mathds{Z})$ of integer-valued harmonic 1-cochains on $\mathcal{BT}^{r}$. This also gives rise to a 
	natural $\mathds{Z}$-structure on the first ($\ell$-adic or de Rham) cohomology of $\Omega^{r}$.
\end{abstract}

\setcounter{section}{-1}

\section{Introduction}

The non-archimedean symmetric spaces $\Omega = \Omega^{r}$ introduced by Drinfeld \cite{Drinfeld74} have shown great importance in the theories of modular and automorphic forms and of Shimura varieties, in the analytic uniformization
of algebraic varieties, in the representation theory of $\mathrm{GL}(r,K)$, in the local Langlands correspondence, and in several other topics of the arithmetic of non-archimedean local fields $K$. An incomplete list of a few references 
is \cite{ManinDrinfeld73}, \cite{Mustafin78}, \cite{GerritzenVanDerPut80}, \cite{SchneiderStuhler91}, \cite{Laumon96}, \cite{DeShalit01}.

For a complete non-archimedean local field $K$ with finite residue class field $\mathds{F}$ and completed algebraic closure $C$, the space $\Omega$ is defined as the complement of the $K$-rational hyperplanes in $\mathds{P}^{r-1}(C)$.
It carries a natural structure as a rigid-analytic space defined over $K$, and is supplied with an action of the group $\mathrm{PGL}(r,K)$. In contrast with the case of real symmetric spaces, it fails to be simply connected (in the étale topology, see \cite{FresnelVanDerPut04}, pp. 160/161), but has
a rich cohomological structure. Its cohomology (for cohomology theories satisfying some natural axioms) has been calculated by Schneider and Stuhler \cite{SchneiderStuhler91}, see also \cite{DeShalit01} and \cite{IovitaSpiess01}.

Suppose for the moment that $\boldsymbol{r=2}$. In this case, $\Omega = \Omega^{2}$ has dimension 1, and a coarse combinatorial picture is provided by the Bruhat-Tits tree $\mathcal{T}$ of $\mathrm{PGL}(2,K)$, a $(q+1)$-regular tree, where
$q = \#(\mathds{F})$ is the residue class cardinality of $K$. A map $\varphi$ from the set $\mathbf{A}(\mathcal{T})$ of oriented 1-simplices (\enquote{arrows}) of $\mathcal{T}$ to $\mathds{Z}$ that satisfies
\begin{enumerate}[label=(\Alph*)]
	\item $\varphi(e) + \varphi(\bar{e}) = 0$ for each $e \in \mathbf{A}(\mathcal{T})$ with inverse $\bar{e}$, and
	\item $\sum \varphi(e) = 0$ for each vertex $v$ of $\mathcal{T}$, where $e$ runs through the arrows emanating from $v$,
\end{enumerate}
is called a $(\mathds{Z}$-valued) \textbf{harmonic cochain} on $\mathcal{T}$. The group $\mathbf{H}(\mathcal{T}, \mathds{Z})$ of all such yields upon tensoring with $\mathds{Z}_{\ell}$ ($\ell$ a prime coprime with $q$) 
the first étale cohomology group $H_{\acute{e}t}^{1}(\Omega^{2}, \mathds{Z}_{\ell})$ of $\Omega^{2}$ (\cite{Drinfeld74} Proposition 10.2). In 1981 Marius van der Put (\cite{VanDerPut8182}, see also \cite{FresnelVanDerPut81} I.8.9) established a short exact sequence
\begin{equation} \label{Eq.van-der-Put-Short-exact-sequence-of-PGL-2-K-modules}
	\begin{tikzcd}
		1 \arrow[r] & C^{*} \arrow[r] & \mathcal{O}(\Omega^{2})^{*} \ar[r, "P"] & \mathbf{H}(\mathcal{T}, \mathds{Z}) \ar[r] & 0 \tag{0.1}
	\end{tikzcd}
\end{equation}
of $\mathrm{PGL}(2,K)$-modules, where $\mathcal{O}(\Omega^{2})$ is the $C$-algebra of holomorphic functions on $\Omega^{2}$ with multiplicative group $\mathcal{O}(\Omega^{2})^{*}$. The van der Put transform $P(u)$ of an invertible function
$u$ is a substitute for the logarithmic derivative $u'/u$, and \eqref{Eq.van-der-Put-Short-exact-sequence-of-PGL-2-K-modules} provides the starting point for a study of the \enquote{Riemann surface} $\Gamma \setminus \Omega^{2}$, where
$\Gamma \subset \mathrm{PGL}(2,K)$ is a discrete subgroup (\cite{GerritzenVanDerPut80}, \cite{Gekeler96}).

It is the aim of the present paper to develop a higher-rank (i.e., $r>2$) analogue of \eqref{Eq.van-der-Put-Short-exact-sequence-of-PGL-2-K-modules}. In \cite{GekelerTA} it was shown that the absolute value $\lvert u \rvert$ of $u \in \mathcal{O}(\Omega^{r})^{*}$
factors over the building map
\[
	\lambda \colon \Omega^{r} \longrightarrow \mathcal{BT}^{r}
\]
and that its logarithm $\log_{q} \lvert u \rvert$ defines an affine map on $\mathcal{BT}^{r}(\mathds{Q})$. Here $\mathcal{BT}^{r}$ is the Bruhat-Tits building of $\mathrm{PGL}(r,K)$ (the higher-dimensional analogue of $\mathcal{BT}^{2} = \mathcal{T}$)
and $\mathcal{BT}^{r}(\mathds{Q})$ is the set of $\mathds{Q}$-points of its realization $\mathcal{BT}^{r}(\mathds{R})$. This makes it feasible that $u \mapsto \log_{q} \lvert u \rvert$ gives rise to a construction of $P$ generalizing van der Put's in the
case $r=2$. The transform $P(u)$ of $u$ will be a $\mathds{Z}$-valued function on the set of arrows $\mathbf{A}(\mathcal{BT}^{r})$ of $\mathcal{BT}^{r}$ subject to (obvious generalizations of) the conditions (A) and (B) above.

Our first result, Proposition \ref{Proposition.Property-C-holds-for-invertible-holomorphic-functions-on-Omega}, is that $P(u)$ satisfies one more relation (condition (C) in Corollary \ref{Corollary.Condition-C}) not visible if $r=2$. 
We then define $\mathbf{H}(\mathcal{BT}^{r}, \mathds{Z})$ as the group of those $\varphi \colon \mathbf{A}(\mathcal{BT}^{r}) \to \mathds{Z}$ which satisfy (A), (B) and (C).

The principal result of the present paper is the fact that the set of these relations is complete:

\begin{manualtheorem}{3.11}
	The map $P \colon \mathcal{O}(\Omega^{r})^{*} \to \mathbf{H}(\mathcal{BT}^{r}, \mathds{Z})$ is surjective, and the van der Put sequence
	\begin{equation} \label{Eq.van-der-Put-Short-exact-sequence-for-general-r}
		\begin{tikzcd}
			1 \arrow[r] & C^{*} \arrow[r] & \mathcal{O}(\Omega^{r})^{*} \ar[r] & \mathbf{H}(\mathcal{BT}^{r}, \mathds{Z}) \ar[r] & 0 \tag{0.2}
		\end{tikzcd}
	\end{equation}
	is an exact sequence of $\mathrm{PGL}(r,K)$-modules.
\end{manualtheorem}

The proof requires the construction of certain functions $u = f_{H,H',n}$ whose transforms $P(u)$ have a prescribed behavior on the finite subcomplex $\mathcal{BT}^{r}(n)$ of $\mathcal{BT}^{r}$, and a crucial technical result 
(Proposition \ref{Proposition.Crucial-technical-tool-for-main-theorem}), which solely refers to the geometry of $\mathcal{BT}^{r}$.

Still, $\mathbf{H}(\mathcal{BT}^{r}, \mathds{Z})$ is a torsion-free abelian group of complicated appearance. However, as a further consequence of Proposition \ref{Proposition.Crucial-technical-tool-for-main-theorem}, we are able to describe 
it in Theorem \ref{Theorem.Abstraction-to-arbitrary-abelian-groups}
\begin{itemize}
	\item either as $\mathbf{H}(\mathcal{T}_{v_{0}}, \mathds{Z})$, where $\mathcal{T}_{v_{0}}$ is a subcomplex of dimension 1 of $\mathcal{BT}^{r}$ (in fact, a tree, which for $r=2$ agrees with the Bruhat-Tits tree $\mathcal{T} = \mathcal{BT}^{2}$),
	and where only conditions (A) and (B) are involved,
	\item or as the group $\mathbf{D}^{0}(\mathds{P}(V), \mathds{Z})$ of $\mathds{Z}$-valued distributions of total mass 0 on the projective space $\mathds{P}(V)$, or by duality, on the 
	compact space $\mathds{P}(V^{\wedge})$ of hyperplanes of the $K$-vector space $V=K^{r}$.
\end{itemize}

As the corresponding group $\mathbf{D}^{0}(\mathds{P}(V^{\wedge}), A)$ with coefficients in some ring $A$ depending on the cohomology theory used (e.g., $A = \mathds{Z}_{\ell}$ for étale cohomology, or,
in characteristic zero, $A=K$ for de Rham cohomology)
has been shown to agree with the first cohomology $H^{1}(\Omega^{r}, A)$ (\cite{SchneiderStuhler91}, Section 3, Theorem 1), we get in some cases a natural integral structure on $H^{1}(\Omega^{r}, A)$ along with a concrete arithmetic interpretation.

\textbf{Acknowledgement:} The Author is grateful to the referees for their accurate reading and helpful comments. These resulted in a number of improvements and corrections and enhanced the paper's readability.
He also wishes to thank Lennart Gehrmann for pointing out an error in Proposition \ref{Proposition.Canonical-maps} of an earlier version of the paper.

\section{Background}

\subsection{} Throughout, $K$ denotes a non-archimedean local field with ring $O$ of integers, a fixed uniformizer $\pi$, and finite residue class field $O/(\pi) = \mathds{F} = \mathds{F}_{q}$ of cardinality $q$. Hence $K$ is a finite extension of either
a $p$-adic field $\mathds{Q}_{p}$ or of a Laurent series field $\mathds{F}_{p} ((X))$. We normalize its absolute value $\lvert \cdot \rvert$ by $\lvert \pi \rvert = q^{-1}$, and let $C = \widehat{\bar{K}}$ be its completed algebraic closure with respect to the unique extension of $\lvert \cdot \rvert$ to $\bar{K}$. The ring of integers of $C$ and its maximal ideal are denoted by $O_{C}$ and $\mathfrak{m}_{C}$. Note that the 
residue class field $O_{C}/\mathfrak{m}_{C}$ is an algebraic closure $\bar{\mathds{F}}$ of $\mathds{F}$. Further, $\log \colon C^{*} \to \mathds{Q}$ is the map $z \mapsto \log_{q} \lvert z \rvert$.

\subsection{} Given a natural number $r \geq 2$, the Drinfeld symmetric space $\Omega = \Omega^{r}$ of dimension $r-1$ is the complement $\Omega = \mathds{P}^{r-1} \setminus \bigcup H$ of the $K$-rational hyperplanes $H$
in projective space $\mathds{P}^{r-1}$. Hence the set of $C$-valued points of $\Omega$ (for which we briefly write $\Omega$) is
\begin{align*}
	\Omega = \{ (\omega_{1} : \ldots : \omega_{r}) \in \mathds{P}^{r-1}(C) \mid \text{the $\omega_{i}$ are $K$-linearly independent}\}.
\end{align*}
If not indicated otherwise, we always suppose that projective coordinates \mbox{$(\omega_{1} : \ldots : \omega_{r})$} are \textbf{unimodular}, that is $\max_{i} \lvert \omega_{i} \rvert = 1$. The set $\Omega$ carries a natural structure 
as a rigid-analytic space defined over $K$ (see \cite{Drinfeld74}, \cite{DeligneHusemoller87}, \cite{SchneiderStuhler91}); in fact, it is an admissible open subspace of $\mathds{P}^{r-1}$, and even a Stein domain (\cite{SchneiderStuhler91}, Section 1, Proposition 14; see \cite{Kiehl67} for the notion of non-archimedean Stein domain).

\subsection{} Let $G$ be the group scheme $\mathrm{GL}(r)$ with center $Z$; hence $G(K) = \mathrm{GL}(r,K)$, $Z(K) \cong K^{*}$, etc. The Bruhat-Tits building \cite{BruhatTits72} $\mathcal{BT} = \mathcal{BT}^{r}$ of 
$G(K)/Z(K) = \mathrm{PGL}(r,K)$ is a contractible simplicial complex with set of vertices
\begin{equation}
	\mathbf{V}(\mathcal{BT}) = \{ [L] \mid L \text{ an $O$-lattice in $V$}\},
\end{equation}
where $L$ runs through the set of $O$-lattices in the $K$-vector space $V = K^{r}$ and $[L]$ is the similarity class of $L$. (An \textbf{$\boldsymbol{O}$-lattice} is a free $O$-submodule of rank $r$ of $V$, two such, $L$ and $L'$, are \textbf{similar}
if there exists $0 \neq c \in K$ such that $L' = cL$.) The classes $[L_{0}], \dots, [L_{s}]$ form an $s$-simplex if and only if they are represented by lattices $L_{i}$ such that
\begin{equation}
	L_{0} \supsetneq L_{1} \supsetneq \dots \supsetneq L_{s} \supsetneq \pi L_{0}.
\end{equation}
The \textbf{combinatorial distance} $d(v,v')$ of two vertices $v,v' \in \mathbf{V}(\mathcal{BT})$ is the length of a shortest path connecting them in the 1-skeleton of $\mathcal{BT}$. It is easily verified that 
\begin{equation}
	d(v,v') = \min \left\{ n \, \middle\vert \, \begin{array}{l} \text{$\exists$ representatives $L,L'$ for $v,v'$} \\ \text{such that $L \supset L' \supset \pi^{n} L$} \end{array} \right\}.
\end{equation}
The \textbf{star} $\mathrm{st}(v)$ of $v \in \mathbf{V}(\mathcal{BT})$ will always denote the full subcomplex of $\mathcal{BT}$ with set of vertices
\begin{equation} \label{Eq.Ster-of-vertices}
	\mathbf{V}(\mathrm{st}(v)) = \{ w \in \mathbf{V}(\mathcal{BT}) \mid d(v,w) \leq 1 \}.
\end{equation}
We regard $V$ as a space of row vectors, on which $G(K)$ acts as a matrix group from the right. Hence $G(K)$ acts also from the right on $\mathcal{BT}$. If the syntax requires a left action, we shift this action to the left by the usual
formula $\gamma x \vcentcolon= x\gamma^{-1}$.

\subsection{} The relationship between $\Omega$ and $\mathcal{BT}$ is as follows: By the Goldman-Iwahori theorem \cite{GoldmanIwahori63}, the realization $\mathcal{BT}(\mathds{R})$ of $\mathcal{BT}$ is in a natural one-to-one 
correspondence with the set of similarity classes of real-valued non-archimedean norms on $V$, where a vertex $v = [L] \in \mathbf{V}(\mathcal{BT}) = \mathcal{BT}(\mathds{Z})$ corresponds to the class of a norm with unit ball $L \subset V$. 
Now the \textbf{building map}
\begin{equation} \label{Equation.Building-map}
	\begin{split}		
							\lambda \colon \Omega 							&\longrightarrow \mathcal{BT}(\mathds{R}) \\
		\boldsymbol{\omega} = (\omega_{1}:\ldots:\omega_{r}) 	&\longmapsto [\nu_{\boldsymbol{\omega}}] 
	\end{split}
\end{equation} 
is well-defined, where the norm $\nu_{\boldsymbol{\omega}}$ maps $\mathbf{x} = (x_{1}, \dots, x_{r}) \in V$ to 
\[
	\nu_{\boldsymbol{\omega}}(\mathbf{x}) = \bigg\lvert \sum_{1 \leq i \leq r} x_{i} \omega_{i} \bigg\rvert,
\] 
and $[\nu_{\boldsymbol{\omega}}]$ is its 
similarity class. Since the value group is $\lvert C^{*} \rvert = q^{\mathds{Q}}$, $\lambda$ maps to $\mathcal{BT}(\mathds{Q})$, and is in fact onto $\mathcal{BT}(\mathds{Q})$, the set of points of $\mathcal{BT}(\mathds{R})$ with 
rational barycentric coordinates. $G(K)$ acts from the left on the set of norms via
\begin{equation}
	\gamma \nu(\mathbf{x}) \vcentcolon= \nu(\mathbf{x}\gamma)
\end{equation}
for $\mathbf{x} \in V$, a norm $\nu$, and $\gamma \in G(K)$; the reader may verify that $\lambda$ is $G(K)$-equivariant, where the action on $\Omega$ is the standard one through left matrix multiplication.
The pre-images under $\lambda$ of simplices of $\mathcal{BT}$ yield an admissible covering of $\Omega$, see e.g. \cite{DeShalit01} (6.2) and (6.3). We therefore consider $\mathcal{BT}$ as a combinatorial picture of $\Omega$.

We cite the following results from \cite{Gekeler17} and \cite{GekelerTA}.

\begin{Theorem}[\cite{GekelerTA} Theorem 2.4] \label{Theorem.Image-of-invertible-holomorphic-function}
	Let $u$ be an invertible holomorphic function on $\Omega$. Then $\lvert u(\boldsymbol{\omega}) \rvert$ depends only on the image $\lambda(\boldsymbol{\omega})$ of $\boldsymbol{\omega} \in \Omega$ in $\mathcal{BT}(\mathds{Q})$.
\end{Theorem}

\subsubsection{} \label{subsubsec-Spectral-norm} We thus define the \textbf{spectral norm} $\lVert u \rVert_{x}$ as the common absolute value $\lvert u(\boldsymbol{\omega}) \rvert$ for all $\omega \in \lambda^{-1}(x)$, 
where $x \in \mathcal{BT}(\mathds{Q})$.

\begin{Theorem}[\cite{GekelerTA} Theorem 2.6] \label{Theorem.Invertible-holomorphic-functions-are-affine-on-certain-sets}
	Let $u$ be an invertible holomorphic function on $\Omega$. Then $\log u = \log_{q} \lvert u \rvert$ regarded as a function on $\mathcal{BT}(\mathds{Q})$ is affine, that is, interpolates linearly in simplices.
\end{Theorem}

\subsection{} Let $\mathbf{A}(\mathcal{BT})$ be the set of \textbf{arrows}, i.e., of oriented 1-simplices of $\mathcal{BT}$. For each arrow $e = (v,v') = ([L], [L'])$ we write
\begin{multline}
	o(e) = \text{origin of $e$} \vcentcolon = v, \quad t(e) = \text{terminus of $e$} \vcentcolon = v', \\ \text{and } \mathrm{type}(e) \vcentcolon= \dim_{\mathds{F}}(L'/\pi L),
\end{multline}
where $L,L'$ are representatives with $L \supset L' \supset \pi L$. Then $1 \leq \mathrm{type}(e) \leq r-1$ and $\mathrm{type}(e) + \mathrm{type}(\bar{e}) = r$, where $\bar{e} = (v',v)$ is $e$ with reverse orientation.
We let
\begin{equation}
	\mathbf{A}_{v} = \bigcupdot_{1 \leq t \leq r-1} \mathbf{A}_{v,t}
\end{equation}
be the arrows $e$ with $o(e) = v$, grouped according to their types $t$. For an invertible function $u$ on $\Omega$ and an arrow $e = (v,w)$, define the \textbf{van der Put value} $P(u)(e)$ of $u$ on $e$ as
\begin{equation}
	P(u)(e) = \log_{q} \lVert u \rVert_{w} - \log_{q} \lVert u \rVert_{v}
\end{equation}
with the spectral norm of \ref{subsubsec-Spectral-norm}.

\begin{Proposition}[\cite{Gekeler17}, Proposition 2.9] \label{Proposition.Properties-of-the-van-der-Put-transform}
	The \textbf{van der Put transform}
	\begin{align*}
		P(u) \colon \mathbf{A}(\mathcal{BT}) 	&\longrightarrow \mathds{Q} \\
															e		&\longmapsto P(u)(e)
	\end{align*}
	of $u$ has in fact values in $\mathds{Z}$ and satisfies
	\begin{equation} \label{Eq.Property-van-der-Put-transform}
		\sum_{e \in \mathbf{A}_{v,1}} P(u)(e) = 0
	\end{equation}
	for all $v \in \mathbf{V}(\mathcal{BT})$. Here the sum is over the arrows $e$ with $o(e) = v$ and $\mathrm{type}(e) = 1$.
\end{Proposition}

\stepcounter{subsubsection}

Actually, in \cite{Gekeler17} the condition $\sum_{e \in \mathbf{A}_{v,r-1}} P(u)(e) = 0$ is given instead of \eqref{Eq.Property-van-der-Put-transform}, due to another choice of orientation. We will discuss this in
more detail in \ref{Subsection.Zeroes} and \ref{Remark.Z-valued-harmonic-cochains}(i), which will also show that both conditions are equivalent in our framework.

\begin{Remark}
	\begin{enumerate}[wide=15pt, label=(\roman*)]
		\item In the case $r=2$, the results \ref{Theorem.Image-of-invertible-holomorphic-function}, \ref{Theorem.Invertible-holomorphic-functions-are-affine-on-certain-sets}, \ref{Proposition.Properties-of-the-van-der-Put-transform} have been 
		known for quite some time: see \cite{VanDerPut8182} and e.g. \cite{FresnelVanDerPut81} I.8.9. For general $r$, they are shown in \cite{Gekeler17} and \cite{GekelerTA} in the framework of these papers, where 
		$\mathrm{char}(K) = \mathrm{char}(\mathds{F}) = p$. However, the proofs make no use of this assumption, and are therefore valid for $\mathrm{char}(K) = 0$, too.
		\item The three cited results are local in the sense that they do not require $u$ to be a global unit. If, e.g., $u$ is a holomorphic function without zeroes on the affinoid $\lambda^{-1}(x)$ with $x \in \mathcal{BT}(\mathds{Q})$, then
		$\lvert u(\boldsymbol{\omega}) \rvert$ is constant on $\lambda^{-1}(x)$; if $u$ is invertible on $\lambda^{-1}(\sigma)$ with a closed simplex $\sigma$ of $\mathcal{BT}$, then $\log u$ is affine there, and if $u$ is invertible 
		on $\lambda^{-1}(\mathrm{st}(v))$, where $\mathrm{st}(v)$ is the star of $v \in \mathbf{V}(\mathcal{BT})$ (see \eqref{Eq.Ster-of-vertices}), then $P(u)(e)$ is defined for all $e \in \mathbf{A}_{v}$ and satisfies 
		\eqref{Eq.Property-van-der-Put-transform}.
		\item It is immediate from the definitions that for invertible functions $u$, $u'$ and arrows $e$,
		\begin{equation} \label{Eq.van-der-put-transform-easy-property-1}
			P(u)(e) + P(u)(\bar{e}) = 0,
		\end{equation}
		and more generally
		\begin{equation} \label{Eq.van-der-put-transform-easy-property-2}
			\sum P(u)(e) = 0, \quad \text{if $e$ runs through the arrows of a closed path in $\mathcal{BT}$},
		\end{equation}
		 as well as
		 \begin{equation} \label{Eq.van-der-put-transform-easy-property-3}
		 	P(uu') = P(u) + P(u').
		 \end{equation}
		 Hence the van der Put transform $P \colon u \mapsto P(u)$ is a homomorphism from the multiplicative group $\mathcal{O}(\Omega)^{*}$ of invertible holomorphic functions on $\Omega$ to the additive group of maps $\varphi \colon \mathbf{A}(\mathcal{BT}) \to \mathds{Z}$ that satisfy \eqref{Eq.van-der-put-transform-easy-property-1}, \eqref{Eq.van-der-put-transform-easy-property-2} and \eqref{Eq.Property-van-der-Put-transform}. Moreover, for $\gamma \in G(K)$,
		 \begin{equation} \label{Eq.van-der-put-transform-group-action-property}
		 	P(u)(e\gamma) = P(u \circ \gamma^{-1})(e),
		 \end{equation} 
		 i.e., $\gamma(P(u)) = P(\gamma u) \vcentcolon = P(u \circ \gamma^{-1})$ holds; whence $P$ is $G(K)$-equivariant. 
	\end{enumerate}
\end{Remark}

In Theorem 3.10 we will find exact conditions that characterize the image of $P$. This will yield the exact sequence
\eqref{Eq.van-der-Put-Short-exact-sequence-for-general-r} of $G(K)$-modules that generalizes \eqref{Eq.van-der-Put-Short-exact-sequence-of-PGL-2-K-modules}.

\section{Evaluation of \emph{P} on elementary rational functions}

\subsection{} Let $U$ be a subspace of $V = K^{r}$ of dimension $t$, where $1 \leq t \leq r-1$. We define the shift toward $U$ on $\mathbf{V}(\mathcal{BT})$ by
\begin{equation}
	\begin{split}
		\tau_{U} \colon \mathbf{V}(\mathcal{BT}) 	&\longrightarrow \mathbf{V}(\mathcal{BT}), \\
												v = [L]				&\longmapsto [L']
	\end{split}
\end{equation}
where $L' = (L \cap U) + \pi L$. Obviously, $e = (v, \tau_{U}(v))$ is a well-defined arrow of type $\mathrm{type}(e) = \dim U = t$. We say that \textbf{$e$ points to $U$}.

\stepcounter{subsubsection}
\stepcounter{equation}

\subsubsection{} For a local ring $R$ (in practice: $R=K$, or $O$, or a finite quotient $O_{n} \vcentcolon = O/(\pi^{n})$) and a free $R$-module $F$ of finite rank, let $\mathrm{Gr}_{R,t}(F)$ be the 
Grassmannian of direct summands $F'$ such that $\mathrm{rank}_{R} F' = t$. Fixing $v = [L] \in \mathbf{V}(\mathcal{BT})$, there is a natural surjective map
\begin{equation}
	\begin{split}
		\mathrm{Gr}_{K,t}(V)	&\longrightarrow \mathbf{A}_{v,t} \\
								U		&\longmapsto (v, \tau_{U}(v))
	\end{split}
\end{equation}
and a canonical bijection
\begin{equation} \label{Eq.Canonical-bijection-of-Avt-and-Gr-FtLpiL}
	\begin{tikzcd}
		\mathbf{A}_{v,t} \ar[r, "\cong"] & \mathrm{Gr}_{\mathds{F}, t}(L/\pi L)
	\end{tikzcd}	
\end{equation}
given by $e = (v,w) = ([L],[M]) \mapsto \bar{M} \vcentcolon = M/\pi L$, where $L \supset M \supset \pi L$. We denote the image of $e$ by $\bar{M}_{e}$ and the pre-image of $\bar{M}$ in $\mathbf{A}_{v,t}$ by $e_{\bar{M}}$.
\stepcounter{subsubsection} \stepcounter{subsubsection}
\subsubsection{} For two arrows $e=e_{\bar{M}}$ and $e' = e_{\bar{M}'}$ with the same origin, we write $e \prec e'$ ($e'$ \textbf{dominates} $e$) if and only if $\bar{M} \subset \bar{M}'$.
\stepcounter{equation} \stepcounter{equation}
\subsubsection{} \label{subsubsection.Identification-of-arrows-with-special-Grassmannians} Fix $n \in \mathds{N}$, let $O_{n}$ be the ring $O/(\pi^{n})$ and let $t \in \{1,r-1\}$. Then, as a generalization 
of the above, $U \mapsto (v, \tau_{U}(v), \dots, \tau_{U}^{n}(v))$ is surjective from $\mathrm{Gr}_{K,t}(V)$ onto the set $\mathbf{A}_{v,t,n}$ of paths of length $n$ in $\mathcal{BT}$ which emanate from $v$, 
are composed of arrows of type $t$, and whose endpoints $w$ have distance $d(v,w) = n$ (e.g., $\mathbf{A}_{v,t,1} = \mathbf{A}_{v,t}$). The set $\mathbf{A}_{v,t,n}$ corresponds one-to-one to 
$\mathrm{Gr}_{O_{n}, t}(L/\pi^{n}L)$, where the composite map from $\mathrm{Gr}_{K,t}(V)$ to $\mathrm{Gr}_{O_{n}, t}(L/\pi^{n}L)$ is given by $U \mapsto ( (L\cap U) + \pi^{n} L)/\pi^{n} L$. This yields in the 
limit the canonical bijections
\begin{equation}
	\begin{tikzcd}
		\mathrm{Gr}_{K,t}(V) \ar[r, "\cong"]	& \varprojlim\limits_{n} \mathbf{A}_{v,t,n} = \varprojlim\limits_{n} \mathrm{Gr}_{O_{n},t} (L/\pi^{n} L) = \mathrm{Gr}_{O,t}(L),
	\end{tikzcd}
\end{equation}
whose composition is simply $U \mapsto U \cap L$. Let $e$ be an arrow of type $t$. Then
\begin{equation} \label{Eq.Special-grassmannian-is-compact-and-open-in-Gr-KtV}
	\mathrm{Gr}_{K,t}(e) \defeq \{ U \in \mathrm{Gr}_{K,t}(V) \mid e \text{ points to } U \}
\end{equation}
is compact and open in the compact space $\mathrm{Gr}_{K,t}(V)$, and it follows from the considerations above that the set of all $\mathrm{Gr}_{K,t}(e)$, where $v$ is fixed and $e$ belongs to 
$\mathbf{A}_{v,t,n}$ for some $n \in \mathds{N}$, form a basis for the topology on $\mathrm{Gr}_{K,t}(V)$.
\subsection{} Given a hyperplane $H$ in $V$, we let $\ell_{H} \colon V \to K$ be a linear form with kernel $H$. We denote by the same symbol its extension $\ell_{H} \colon V \otimes_{K} C = C^{r} \to C$. The quotients
\begin{equation}
	\ell_{H,H'} \vcentcolon = \ell_{H}/\ell_{H'}
\end{equation}
of two such are rational functions on $\mathds{P}^{r-1}(C)$ without zeroes or poles on $\Omega \hookrightarrow \mathds{P}^{r-1}(C)$. Note that $\ell_{H}$ is determined up to multiplication by a non-zero 
scalar in $K$; hence $P(\ell_{H,H'})$ depends only on $H$ and $H'$, but not on the scaling of $\ell_{H}$ and $\ell_{H'}$. Our first task will be to describe $P(\ell_{H,H'})$.
\subsection{} We start with a closer look to the building map $\lambda$. Let $v_{0} = [L_{0}]$ be the standard vertex, where $L_{0}$ is the standard lattice $O^{r}$ in $V$. Let us first recall the easily verified fact 
(where the unimodularity normalization of $\boldsymbol{\omega} \in \Omega$ is used): 
\begin{align}
	\Omega_{v_{0}} \defeq \lambda^{-1}(v_{0}) 	&= \{ \boldsymbol{\omega} \in \Omega \mid \nu_{\boldsymbol{\omega}} \text{ has unit ball } L_{0} \} \\
																		&= \{ \boldsymbol{\omega} \in \Omega \mid \text{ the $\omega_{i}$ are orthogonal and $\lvert \omega_{i} \rvert = 1$ for $1 \leq i \leq r$} \}. \nonumber
\end{align}
($z_{1}, \dots, z_{n} \in C$ are \textbf{orthogonal} if and only if $\lvert \sum_{1 \leq i \leq r} a_{i} z_{i} \rvert = \max_{i} \lvert a_{i} z_{i} \rvert$ for arbitrary coefficients $a_{i} \in K$.) Hence the canonical reduction of 
$\Omega_{v_{0}}$ equals
\begin{equation} \label{Eq.Canonical-reduction-of-preimage-of-v0-under-lambda}
	\bar{\Omega}_{v_{0}} = (\mathds{P}^{r-1}/\mathds{F}) \setminus \bigcup \bar{H},
\end{equation}
where $\bar{H}$ runs through the hyperplanes defined over $O/(\pi) = \mathds{F}$. A similar description holds for $\overline{\lambda^{-1}(v)}$ if $v$ is an arbitrary vertex, but we need some preparations.
\subsection{} Write $\langle \cdot , \cdot \rangle$ for the standard bilinear form on $V$ given by 
\[
	\langle \mathbf{x}', \mathbf{x} \rangle = \sum_{1 \leq i \leq r} x_{i}'x_{i},
\]
which we extend to a form $\langle \cdot, \cdot \rangle$ on $C^{r}$. It identifies $V = K^{r}$ with its dual $V^{\wedge}$. For each $K$-subspace $U$ of $V$, let
\begin{equation} \label{Eq.Orthogonal-complement}
	U^{\perp} \defeq \{ \mathbf{x} \in V \mid \langle \mathbf{x}, U \rangle = 0 \}
\end{equation}
be its orthogonal with respect to $\langle \cdot, \cdot \rangle$. For an $O$-lattice $L$ in $V$,
\begin{equation}
	L^{\wedge} \defeq \{ \mathbf{x} \in V \mid \langle \mathbf{x}, L  \rangle \subset O \}
\end{equation}
is the dual lattice. We put $\tilde{\Omega}$ for the pre-image of $\Omega$ in $C^{r}$. Then $L^{\wedge} \otimes_{O} O_{C}$ embeds into $C^{r}$, and by \eqref{Equation.Building-map} we find:
\stepcounter{subsubsection} \stepcounter{subsubsection} \stepcounter{equation}
\subsubsection{} The image of $(L^{\wedge} \otimes O_{C}) \cap \tilde{\Omega}$ in $\Omega$ equals $\Omega_{v} \defeq \lambda^{-1}(v)$, where $v = [L]$ is the vertex of $\mathcal{BT}$ corresponding
to $L$.

Similarly,
\begin{equation}
	\begin{tikzcd}
		(L^{\wedge} \otimes O_{C}) \otimes O_{C}/\mathfrak{m}_{C} \ar[r, "\cong"]	& (L^{\wedge}/\pi L^{\wedge}) \otimes_{\mathds{F}} \bar{\mathds{F}} = (L/\pi L)^{\wedge} \otimes_{\mathds{F}} \mathds{F}.
	\end{tikzcd}
\end{equation}
\subsection{} Let $\mathds{P}(L^{\wedge}/\pi L^{\wedge})/\mathds{F}$ be the projective space attached to the $r$-dimensional vector space $L^{\wedge}/\pi L^{\wedge}$, regarded as a scheme over
$O/(\pi) = \mathds{F}$. Its $\mathds{F}$-rational hyperplanes correspond to those of the vector space $(L^{\wedge}/\pi L^{\wedge}) \otimes_{\mathds{F}} \bar{\mathds{F}}$, or, by duality, to the
$\mathds{F}$-lines (one-dimensional $\mathds{F}$-subspaces) $\bar{G}$ in $L/\pi L$. Therefore, the canonical reduction $\bar{\Omega}_{v}$ of $\Omega_{v}$ is
\begin{equation} \label{Eq.Description-for-Omega-bar-v}
	\bar{\Omega}_{v} = (\mathds{P}( L^{\wedge}/ \pi L^{\wedge})/\mathds{F}) \setminus \bigcup \bar{H},
\end{equation}
where $\bar{H}$ runs through the hyperplanes defined over $\mathds{F}$. The set of these is in canonical bijection with the set of $\mathds{F}$-lines in $L/\pi L$, that is, with $\mathbf{A}_{v,1}$. For each
$e \in \mathbf{A}_{v,1}$ let $\bar{H}_{e}$ be the corresponding $\mathds{F}$-hyperplane in \eqref{Eq.Description-for-Omega-bar-v}.

\stepcounter{subsubsection}

\subsubsection{} The object $\bar{H}_{e}$ (an $\mathds{F}$-subspace of $L^{\wedge}/\pi L^{\wedge}$ or the corresponding hyperplane in $\mathds{P}(L^{\wedge}/\pi L^{\wedge})/\mathds{F}$, described 
through the same symbol) mustn't be confused with the $\bar{M}_{e}$ of \eqref{Eq.Canonical-bijection-of-Avt-and-Gr-FtLpiL}, which is an $\mathds{F}$-subspace of $L/\pi L$. The relationship is as follows. 
The form $\langle \cdot, \cdot \rangle$ induces an $\mathds{F}$-bilinear form
\[
	\bar{\langle \cdot, \cdot \rangle} \colon L/\pi L \times L^{\wedge}/\pi L^{\wedge} \longrightarrow \mathds{F}.
\]
For an $\mathds{F}$-subspace $\bar{M}$ of $L/\pi L$, $\bar{M}^{\perp}$ denotes its orthogonal with respect to $\bar{\langle \cdot, \cdot \rangle}$ in $L^{\wedge}/\pi L^{\wedge}$. Let $\bar{G}_{e} \subset L/\pi L$
be the line defined by $e \in \mathbf{A}_{v,1}$ as in \eqref{Eq.Canonical-bijection-of-Avt-and-Gr-FtLpiL}. Then $\bar{H}_{e} = (\bar{G}_{e})^{\perp}$.

\subsection{} \label{Subsection.Zeroes} Let $u$ be an invertible holomorphic function on $\Omega$, scaled such that $\lVert u \rVert_{v} = 1$. Its reduction $\bar{u}$ at $v$ is a rational function
on $\bar{\Omega}_{v}$ without zeroes or poles. For each $e \in \mathbf{A}_{v,1}$ let $m_{e}$ be the vanishing order of $\bar{u}$ along $\bar{H}_{e}$ (negative, if $\bar{u}$ has a pole along $\bar{H}_{e})$,
and let $\ell_{e}$ be a linear form on $\mathds{P}( L^{\wedge}/\pi L^{\wedge})/\mathds{F}$ with vanishing locus $\bar{H}_{e}$. Up to a multiplicative constant, $\bar{u}$ equals 
$\prod_{e \in \mathbf{A}_{v,1}} \ell_{e}^{m_{e}}$, and so 
\begin{equation}
	\sum_{e \in \mathbf{A}_{v,1}} m_{e} = \text{weight of the form $\bar{u}$} = 0.
\end{equation}
Now the value of the van der Put transform on $e \in \mathbf{A}_{v,1}$ is (with notation above)
\begin{equation} \label{Equation.Multiplicity}
	P(u)(e) = -m_{e}.
\end{equation}
To see this, we may assume (by \eqref{Eq.van-der-put-transform-group-action-property}, and since the action of $G(K)$ on arrows of type 1 is transitive) that $e = (v_{0}, v_{1})$ with $v_{0} = [L_{0}]$, $v_{1} = [L_{1}]$, $L_{0} = O^{r}$, $L_{1} = (\pi) \times \dots \times (\pi) \times O$. Then $L_{0}^{\wedge} = L_{0}$, $L_{1}^{\wedge} = (\pi^{-1}) \times \dots \times (\pi^{-1}) \times O$, $\bar{H}_{e}$ is the hyperplane $\{ (*:\ldots:*:0) \}$ in 
$\mathds{P}(L_{0}^{\wedge}/\pi L_{0}^{\wedge})/\mathds{F} = \mathds{P}^{r-1}/\mathds{F}$, and we may choose $\ell_{e} \colon \mathbf{x} = (x_{1}: \ldots : x_{r}) \mapsto x_{r}$, which is the reduction of the global
form $\tilde{\ell} \colon \boldsymbol{\omega} = (\omega_{1}: \ldots : \omega_{r}) \mapsto \omega_{r}$ on $\Omega$. In order to get \enquote{functions} instead of \enquote{forms}, we work with
$\ell_{e}/\ell_{1}$ (resp. $\tilde{\ell}/\tilde{\ell}_{1})$, where $\ell_{1} \colon \mathbf{x} \mapsto x_{1}$ with lift $\tilde{\ell}_{1} \colon \boldsymbol{\omega} \mapsto \omega_{1}$. If $\bar{u}$ has a zero of order
$m$ along $\bar{H}_{e}$ (i.e., $\bar{u} = \bar{u}_{0}(\ell_{e}/\ell_{1})^{m}$, where $\bar{u}_{0}$ has neither zeroes nor poles on $\bar{\Omega}_{v_{0}}$ and along $\bar{H}_{e}$), then $u$ grows like
$(\tilde{\ell}/\tilde{\ell}_{1})^{m}$ when moving from $\Omega_{v_{0}}$ to $\Omega_{v_{1}}$. But the absolute value of $\tilde{\ell}/\tilde{\ell}_{1}$ on $\Omega_{v_{0}}$ is 1, while it is $\lvert \pi \rvert = q^{-1}$ on
$\Omega_{v_{1}}$, which shows \eqref{Equation.Multiplicity}.

Finally, combining the above yields
\begin{equation}
	\sum_{e \in \mathbf{A}_{v,1}} P(u)(e) = 0,
\end{equation}
that is, the assertion of \eqref{Eq.Property-van-der-Put-transform}.

\subsection{} Each hyperplane $H$ of $V$ is given as the kernel of a linear form
\begin{equation}
	\ell_{H} = \ell_{\mathbf{y}} \colon \mathbf{x} \longmapsto \langle \mathbf{y}, \mathbf{x} \rangle
\end{equation}
with some $\mathbf{y} \in L_{0} \setminus \pi L_{0}$. Let $G = K\mathbf{y} = H^{\perp}$ be the line spanned by $\mathbf{y}$. The arrow $(v_{0}, \tau_{G}(v_{0})) \in \mathbf{A}_{v_{0},1}$ equals $e_{\bar{G}}$ with
\[
	\bar{G} = (O\mathbf{y} + \pi L_{0})/\pi L_{0}.
\]
Two such vectors $\mathbf{y}, \mathbf{y}'$ give rise to the same $e_{\bar{G}}$ if and only if $\mathbf{y}' \equiv c \cdot \mathbf{y} \pmod{\pi}$ with some unit $c \in O^{*}$. More generally, $\mathbf{y}$ and $\mathbf{y}'$ give rise 
to the same path $(v_{0}, \tau_{G}(v_{0}), \dots, \tau_{G}^{n}(v_{0})) \in \mathbf{A}_{v_{0},1,n}$ if and only if
\begin{equation} \label{Eq.n-equivalence-of-vectors}
	\mathbf{y}' \equiv c \cdot \mathbf{y} \pmod{\pi^{n}}
\end{equation}
with $c \in O^{*}$. In this case we call $\mathbf{y}$ and $\mathbf{y}'$ \textbf{$\boldsymbol{n}$-equivalent}; the respective equivalence classes are briefly the \textbf{$\boldsymbol{n}$-classes} of $\mathbf{y}, \mathbf{y}'$.

\subsection{} \label{subsection.Rational-function-corresponding-to-two-hyperplanes} Let now hyperplanes $H, H'$ of $V$ be given by $\mathbf{y}, \mathbf{y}'$ as above. Put $G = H^{\perp} = K\mathbf{y}$,
$G' = K\mathbf{y}'$. The function $\ell_{H,H'} = \ell_{\mathbf{y}}/ \ell_{\mathbf{y}'}$ has constant absolute value 1 on $\Omega_{v_{0}}$ and therefore, by reduction, gives a rational function $\bar{\ell}_{H,H'}$
without zeroes or poles on $\bar{\Omega}_{v_{0}} \hookrightarrow \mathds{P}^{r-1}/\mathds{F}$. Put
\[
	\bar{H} = ( (L_{0} \cap H) + \pi L_{0}) / \pi L_{0},
\]
and ditto $\bar{H}'$. By definition, it is an $\mathds{F}$-subvector space of $L_{0}/\pi L_{0} \overset{\cong}{\rightarrow} \mathds{F}^{r}$. As usual, we denote by the same symbol the corresponding 
$\mathds{F}$-rational linear subvariety of $\mathds{P}^{r-1}/\mathds{F}$ that appears e.g. in \eqref{Eq.Canonical-reduction-of-preimage-of-v0-under-lambda}. Suppose that $\bar{H}$ differs from 
$\bar{H}'$. Then $\bar{\ell}_{H,H'}$ has vanishing order 1 along $\bar{H}$, vanishing order $-1$ along $\bar{H}'$, and vanishing order 0 along the other hyperplanes in the boundary of 
$\bar{\Omega}_{v_{0}}$ (see \eqref{Eq.Canonical-reduction-of-preimage-of-v0-under-lambda}).
If however $\bar{H} = \bar{H}'$, then $\bar{\ell}_{H,H'}$ has neither zeroes nor poles along the boundary (and is therefore constant). According to the recipe discussed in \ref{Subsection.Zeroes}, 
we find the following description.

\begin{Proposition} \label{Proposition.van-der-Put-transform-on-special-rational-function-evaluated-on-special-arrow-of-type-1}
	Let $e$ be an arrow in $\mathbf{A}_{v_{0},1}$. Then
	\[
		P(\ell_{H,H'})(e) = \begin{cases} -1,	&e=(v_{0}, \tau_{G}(v_{0})) \neq (v_{0}, \tau_{G'}(v_{0})), \\ +1, 	&e=(v_{0}, \tau_{G'}(v_{0})) \neq (v_{0}, \tau_{G}(v_{0})), \\ \hphantom{+}0,	&\text{otherwise.} \end{cases}
	\]
	~\hfill{$\square$}
\end{Proposition}

Again by the transitivity of the action of $G(K)$, we may transfer \ref{Proposition.van-der-Put-transform-on-special-rational-function-evaluated-on-special-arrow-of-type-1} 
to arbitrary arrows of type 1, and thus get:

\begin{Corollary} \label{Corollary.van-der-Put-transform-of-special-rational-function-evaluated-on-arbitrary-arrow-of-type-1}
	Let $e \in \mathbf{A}_{v,1}$ be an arrow of type 1 with arbitrary origin $v \in \mathbf{V}(\mathcal{BT})$. Write $e_{H}^{\perp}$ (resp. $e_{H'}^{\perp}$) for the arrow $(v, \tau_{H^{\perp}}(v))$ 
	(resp. $(v, \tau_{H'^{\perp}}(v))$). Then
	\[
		P(\ell_{H,H'})(e) = \begin{cases} -1, &e=e_{H}^{\perp} \neq e_{H'}^{\perp}, \\ +1, &e=e_{H'}^{\perp} \neq e_{H}^{\perp}, \\ \hphantom{+}0,	&\text{otherwise.}\end{cases}
	\]
	~\hfill $\square$
\end{Corollary}

Next, we deal with arrows of arbitrary type.

\begin{Proposition} \label{Proposition.van-der-Put-transform-of-special-rational-function-evaluated-on-abitrary-arrow}
	Given hyperplanes $H,H'$ of $V$ and an arrow $e$ of $\mathcal{BT}$ with origin $v \in \mathbf{V}(\mathcal{BT})$, let $e_{H}^{\perp}$ (resp. $e_{H'}^{\perp}$) be the arrow with origin $v$ pointing to 
	$G = H^{\perp}$ (resp. to $G' = H'^{\perp}$). The transform $P(\ell_{H,H'})$ evaluates on $e$ as follows:
	\[
		P(\ell_{H,H'})(e) = \begin{cases} -1, & e_{H}^{\perp} \prec e, e_{H'}^{\perp} \nprec e, \\ +1, &e_{H'}^{\perp} \prec e, e_{H}^{\perp} \nprec e, \\ \hphantom{+}0, &\text{otherwise.} \end{cases}
	\]
\end{Proposition}

\begin{proof}
	Let $L$ be a lattice with $[L] = v$ and $e = e_{\bar{M}}$, where $\bar{M}$ is a subspace of $L/\pi L$ of dimension $t = \mathrm{type}(e)$. Since the case $t=1$ is given by the last corollary, we may 
	assume that $t \geq 2$. Suppose that $e_{H}^{\perp} \prec e$, i.e.,
	\[
		\bar{G} = ( (L \cap G) + \pi L)/\pi L \subset \bar{M} \subset L/\pi L.
	\]	
	Let $\bar{M}_{0} = 0 \subsetneq \bar{M}_{1} = \bar{G} \subsetneq \dots \subsetneq \bar{M}_{t} = \bar{M}$ be a complete flag connecting $0$ to $\bar{M}$, where 
	$\dim \bar{M}_{i} = i$ for $0 \leq i \leq t$. It corresponds to a path $(v_{0}, v_{1}, \dots, v_{t})$ in $\mathcal{BT}$, where $v_{0} = v = [L]$, $v_{t} = t(e_{\bar{M}})$, and all the arrows 
	$e_{1} = (v_{0}, v_{1}), \dots, e_{t} = (v_{t-1}, v_{t})$ of type 1. As $\{v_{0}, \dots, v_{t}\}$ is a $t$-simplex, $d(v_{0}, v_{i}) = 1$ for $1 \leq i \leq t$, and therefore no $e_{i}$ different from 
	$e_{1} = e_{\bar{G}}$ points to $G$.
	
	Suppose that moreover $e_{H'}^{\perp} \nprec e$, that is,
	\[
		( (L\cap G') + \pi L) / \pi L \not\subset \bar{M}.
	\]
	Then none of the $e_{i}$ ($1 \leq i \leq t$) points to $G'$, so
	\[
		P(\ell_{H,H'})(e) = \sum_{1 \leq i \leq t} P(\ell_{H,H'})(e_{i}) = P(\ell_{H,H'})(e_{1}) = -1
	\]
	by \eqref{Eq.van-der-put-transform-easy-property-2} and Corollary \ref{Corollary.van-der-Put-transform-of-special-rational-function-evaluated-on-arbitrary-arrow-of-type-1}. If 
	$e_{H}^{\perp} \neq e_{H'}^{\perp} \prec e$, then we can arrange the flag
	$\bar{M}_{0} \subsetneq \dots \subsetneq \bar{M}_{t}$ such that as before $e_{1}$ points to $G$, $e_{2}$ points to $G'$, and no $e_{i}$ ($3 \leq i \leq t$) points to $G$ or $G'$. In this case
	\[
		P(\ell_{H,H'})(e) = P(\ell_{H,H'})(e_{1}) + P(\ell_{H,H'})(e_{2}) = -1 + 1 = 0.
	\]
	If $e_{H}^{\perp} = e_{H'}^{\perp} \prec e$, then
	\[
		P(\ell_{H,H'})(e) = P(\ell_{H,H'})(e_{1}) = 0 \quad \text{by \ref{Corollary.van-der-Put-transform-of-special-rational-function-evaluated-on-arbitrary-arrow-of-type-1}.}
	\]
	If neither $e_{H}^{\perp} \prec e$ nor $e_{H'}^{\perp} \prec e$, neither of the arrows $e_{i}$ ($1 \leq i \leq t$) corresponding to a flag $\bar{M}_{0} =0 \subsetneq \dots \subsetneq \bar{M}_{t} = \bar{M}$
	points to $G$ or to $G'$, and so $P(\ell_{H,H'})(e) = 0$ results. The case $e_{H'}^{\perp} \prec e$, $e_{H}^{\perp} \nprec e$ comes out by symmetry.
\end{proof}

\begin{Corollary} \label{Corollary.Condition-C}
	Let $H_{1}, \dots, H_{n}$ be finitely many hyperplanes of $V$ with corresponding linear forms $\ell_{i} = \ell_{H_{i}}$, $\mathrm{ker}(\ell_{i}) = H_{i}$, and multiplicities $m_{i} \in \mathds{Z}$ such that 
	$\sum_{1 \leq i \leq n} m_{i} = 0$. The function
	\[
		u \vcentcolon = \prod_{1 \leq i \leq n} \ell_{i}^{m_{i}}
	\]
	is a unit on $\Omega$, whose van der Put transform $P(u)$ satisfies the condition:
	\begin{enumerate}
		\item[$\mathrm{(C)}$] For each arrow $e \in \mathbf{A}(\mathcal{BT})$ with $o(e) = v \in \mathbf{V}(\mathcal{BT})$,
		\[
			P(u)(e) = \sum_{\substack{e' \in \mathbf{A}_{v,1} \\ e' \prec e}} P(u)(e').
		\] 
	\end{enumerate}
\end{Corollary}

\begin{proof}
	(C) is satisfied for $u = \ell_{H,H'} = \ell_{H}/\ell_{H'}$ by \ref{Proposition.van-der-Put-transform-of-special-rational-function-evaluated-on-abitrary-arrow}. The general case follows as condition (C) is linear 
	(it holds for $u \cdot u'$ if it holds for $u$ and $u'$) and
	$\prod \ell_{i}^{m_{i}}$ is a product of functions of type $\ell_{H,H'}$.
\end{proof}

\section{The van der Put sequence}

\begin{Proposition} \label{Proposition.Property-C-holds-for-invertible-holomorphic-functions-on-Omega}
	Let $u$ be an invertible holomorphic function on $\Omega$. Then its van der Put transform $P(u)$ satisfies condition $\mathrm{(C)}$ from Corollary \ref{Corollary.Condition-C}.
\end{Proposition}

\begin{proof}
	Again by \eqref{Eq.van-der-put-transform-group-action-property} we may suppose that the origin $o(e)$ of the arrow in question is equal to $v_{0} = [L_{0}]$. So $e = e_{\bar{M}}$ with some non-trivial 
	$\mathds{F}$-subspace $\bar{M}$ of $L_{0}/\pi L_{0}$. As in \ref{subsection.Rational-function-corresponding-to-two-hyperplanes} we use the same letter $\bar{M}$ for the corresponding linear 
	subvariety of $\mathds{P}^{r-1}/\mathds{F}$ of 
	dimension $t-1$, where $t = \mathrm{type}(e) = \dim \bar{M}$.
	
	Multiplying $u$ by suitable functions of type $\ell_{H,H'}$ (which doesn't alter the (non)-validity of (C) for $u$), we may assume that $P(u)(e') = 0$ for all $e' \in \mathbf{A}_{v_{0},1}$ dominated by $e$. Then we 
	must show that $P(u)(e) = 0$, too. Let $u$ be normalized such that $\lVert u \rVert_{v_{0}} = 1$, and let $\bar{u}$ be its reduction as a rational function on $\mathds{P}^{r-1}/\mathds{F}$, see 
	\eqref{Eq.Canonical-reduction-of-preimage-of-v0-under-lambda}.
	
	If $P(u)(e) < 0$ then $\lvert u \rvert$ decays along $e = e_{\bar{M}}$ and $\bar{u}$ vanishes along $\bar{M}$. Correspondingly, if $P(u)(e) > 0$ then $(\bar{u})^{-1} = (\bar{u^{-1}})$ vanishes along $\bar{M}$. Hence it
	suffices to show that, under our assumptions, $\bar{u}$ restricts to a well-defined rational function on $\bar{M}$, i.e., $\bar{M}$ is neither contained in the vanishing locus $V(\bar{u})$ nor in $V(\bar{u}^{-1})$. But the
	latter is obvious: With a suitable constant $c \neq 0$ we have 
	\[
		\bar{u} = c \cdot \prod \ell_{\bar{H}}^{m(\bar{H})},
	\]
	where $\bar{H}$ runs through the boundary components of $\bar{\Omega}_{v_{0}}$ as in \eqref{Eq.Canonical-reduction-of-preimage-of-v0-under-lambda}, $\ell_{\bar{H}}$ is a linear form 
	vanishing on $\bar{H}$, $\sum m(\bar{H}) = 0$, and $m(\bar{H}) = -P(u)(e_{\bar{H}}^{\perp}) = 0$ if $\bar{H}^{\perp} \subset \bar{M}$. Hence neither the rational function $\bar{u}$ nor its 
	reciprocal vanishes identically on $\bar{M}$.
\end{proof}

\subsection{}\label{subsection.Conditions-for-varphi-in-A(BT)} The proposition motivates the following definition. Let $A$ be any additively written abelian group. The group of 
\textbf{$\boldsymbol{A}$-valued harmonic 1-cochains} $\mathbf{H}(\mathcal{BT}, A)$ is the group of maps $\varphi \colon \mathbf{A}(\mathcal{BT}) \to A$ that satisfy
\begin{enumerate}[label=(\Alph*)]
	\item $\sum \varphi(e) = 0$, whenever $e$ ranges through the arrows of a closed path in $\mathcal{BT}$;
	\item for each type $t$, $1 \leq t \leq r-1$, and each $v \in V(\mathcal{BT})$, the condition 
	\par\noindent\begin{minipage}{\linewidth}
		\begin{equation}
			\sum_{e \in \mathbf{A}_{v,t}} \varphi(e) = 0 \quad \text{holds}; \tag{$\mathrm{B}_{t}$}
		\end{equation}
	\end{minipage}
	\item for each $v \in \mathbf{V}(\mathcal{BT})$ and each $e \in \mathbf{A}_{v}$,
	\[
		\sum_{e' \in \mathbf{A}_{v,1}, ~e' \prec e} \varphi(e') = \varphi(e).
	\]
\end{enumerate}

\begin{Remark} \label{Remark.Z-valued-harmonic-cochains}
	\begin{enumerate}[wide=15pt, label=(\roman*)]
		\item In the case where the coefficient group $A$ equals $\mathds{Z}$, condition (A) is \eqref{Eq.van-der-put-transform-easy-property-2}, $(\mathrm{B}_{1})$ is \eqref{Eq.Property-van-der-Put-transform}, and
		(C) is the condition dealt with in \ref{Corollary.Condition-C} and \ref{Proposition.Property-C-holds-for-invertible-holomorphic-functions-on-Omega}. (A) in particular implies that $\varphi$ is alternating, 
		i.e., $\varphi(\bar{e}) = -\varphi(e)$.
		Further, $(\mathrm{B}_{1})$ together with (C) implies $(\mathrm{B}_{t})$ for all types $t$, as 
		\[
			\sum_{e \in \mathbf{A}_{v,t}} \varphi(e) = \sum_{e' \in \mathbf{A}_{v,1}} \varphi(e') \# \{ e \in \mathbf{A}_{v,t} \mid e' \prec e \},
		\]
		where $\# \{\, \cdots \}$, the cardinality of some finite Grassmannian, is independent of $e'$.
		\item Note that the current $\mathbf{H}(\mathcal{BT}, \mathds{Z})$ differs from the group defined in \cite{Gekeler17}, as condition (C) is absent there.
		\item Proposition \ref{Proposition.Property-C-holds-for-invertible-holomorphic-functions-on-Omega} together with the preceding considerations shows that
		\begin{align*}
			P \colon \mathcal{O}(\Omega)^{*} 	&\longrightarrow \mathbf{H}(\mathcal{BT}, \mathds{Z}) \\
															u		&\longmapsto P(u)
		\end{align*}
		is well-defined. Its kernel consists of the invertible holomorphic functions $f$ on $\Omega$ with constant absolute value, which equals the constants $C^{*}$, as will be shown in 
		\ref{Proposition.Holomorphic-function-on-Omega-constant}. Hence, by \eqref{Eq.van-der-put-transform-group-action-property}, we have the exact sequence of $G(K)$-modules
		\begin{equation*}
			\begin{tikzcd}
				1 \ar[r] &C^{*} \ar[r]	&\mathcal{O}(\Omega)^{*} \ar[r, "P"] &\mathbf{H}(\mathcal{BT}, \mathds{Z}).
			\end{tikzcd}
		\end{equation*}
		In fact, we will show that $P$ is also surjective, which yields our principal result Theorem \ref{Theorem.van-der-Put-sequence-short-exact-sequence-of-GK-modules}.
		\item Beyond the natural coefficient domains $A=\mathds{Z}$ or $\mathds{Q}$ for $\mathbf{H}(\mathcal{BT}, A)$, at least the torsion groups $A = \mathds{Z}/(N)$ deserve interest. For example, 
		in the case $r=2$ and $\mathrm{char}(C) = \mathrm{char}(\mathds{F}) = p$, the invariants $\mathbf{H}(\mathcal{BT}, \mathds{F}_{p})^{\Gamma}$ under an arithmetic subgroup $\Gamma \subset G(K)$ 
		differ in general from $\mathbf{H}(\mathcal{BT}, \mathds{Z})^{\Gamma} \otimes \mathds{F}_{p}$, see \cite{Gekeler96} Section 6. The coefficient rings $A = \mathds{Z}_{\ell}$ ($\ell$ a prime number) 
		and $A=K$ come into the game by relating $\mathbf{H}(\mathcal{BT}, \mathds{Z})$ with the first cohomology of $\Omega$, see Remark 5.5.
	\end{enumerate}
\end{Remark}

\begin{Proposition} \label{Proposition.Holomorphic-function-on-Omega-constant}
	Each bounded holomorphic function $f$ on $\Omega$ is constant. In particular, the kernel of the map $P$ equals the constants $C^{*}$.
\end{Proposition}

\begin{proof}
	Let $\boldsymbol{\omega} = (\omega_{1}:\ldots:\omega_{r})$ be an element of $\Omega$. We are going to show that $f$ as a function in $\omega_{i}$, where 
	$\omega_{1}, \dots, \omega_{i-1}, \omega_{i+1}, \dots, \omega_{r}$ are fixed, is constant for each $i$, which will give the result. Let
	\begin{align*}
		\alpha^{(i)}_{\boldsymbol{\omega}} \colon \mathds{P}^{r-1}(K)	&\longrightarrow \mathds{P}^{1}(C) \qquad \text{be the map} \\
								(x_{1}: \ldots: x_{r})											&\longmapsto \Big( \sum_{\substack{1 \leq j \leq r \\ i \neq j}} x_{j} \omega_{j}: x_{i} \Big).
	\end{align*}
	It is well-defined (since the $\omega_{j}$ are $K$-linearly independent) and continuous with respect to the non-archimedean topologies on both sides, whence its image 
	$\mathrm{im}(\alpha^{(i)}_{\boldsymbol{\omega}})$ is compact. Moreover, the complement 
	$\Omega^{(i)}_{\boldsymbol{\omega}} \defeq \mathds{P}^{1}(C) \setminus \mathrm{im}(\alpha^{(i)}_{\boldsymbol{\omega}})$ in $\mathds{P}^{1}(C)$ equals the set of those 
	$\omega \in C = \{ (\omega:1) \} \hookrightarrow \mathds{P}^{1}(C)$ which are eligible for $(\omega_{1}:\ldots: \omega_{i-1} : \omega : \omega_{i+1} : \ldots : \omega_{r})$ to lie in $\Omega$. Analytic 
	spaces of this shape are extensively discussed in \cite{FresnelVanDerPut04} Chapter II. Notably, their Proposition 2.7.9 states that bounded functions on $\Omega_{\boldsymbol{\omega}}^{(i)}$ are 
	constant as wanted.
\end{proof}

\subsection{} The strategy of proof of the surjectivity of $P$ will be to approximate a given $\varphi \in \mathbf{H}(\mathcal{BT}, \mathds{Z})$ by linear combinations of elements $P(u)$, where $u$ is a function 
of type $\ell_{H,H'}$, or a relative of it.

Given two hyperplanes $H \neq H'$ of $V$ and $n \in \mathds{N}_{0} = \{0,1,2,\dots\}$, define
\begin{equation} \label{Eq.Definition-of-relative-to-ell-H-H'}
	f_{H,H',n} \vcentcolon = 1 + \pi^{n} \ell_{H,H'}.
\end{equation}
Here $\ell_{H,H'} = \ell_{H}/\ell_{H'} = \ell_{\mathbf{y}}/\ell_{\mathbf{y}'}$, where $\mathbf{y}, \mathbf{y}' \in L_{0} \setminus \pi L_{0}$, $H = \mathrm{ker}(\ell_{\mathbf{y}})$, $H' = \mathrm{ker}(\ell_{\mathbf{y}'})$. Like $\ell_{H,H'}$, 
$f_{H,H',n}$ is a unit on $\Omega$. We denote by
\begin{equation}
	\mathcal{BT}(n) \subset \mathcal{BT}
\end{equation}
the full subcomplex with vertices $\mathbf{V}(\mathcal{BT}(n)) = \{ v \in \mathbf{V}(\mathcal{BT}) \mid d(v_{0}, v) \leq n\}$. Hence $\mathcal{BT}(0) = \{v_{0}\}$, $\mathcal{BT}(1) = \mathrm{st}(v_{0})$, etc. Further,
\begin{equation}
	\Omega(n) \defeq \lambda^{-1}(\mathcal{BT}(n)).
\end{equation}
Then $\Omega(n)$ is an admissible affinoid subspace of $\Omega$ and $\Omega = \bigcup_{n \geq 0} \Omega(n)$. (In \cite{SchneiderStuhler91} Section 1, Proposition 4, $\Omega(n)$ is called $\bar{\Omega}_{n}$,
and a system of generators of its affinoid algebra is constructed.)

\begin{Lemma} \label{Lemma.Properties-of-f-H-H'-n-on-Omega-n}
	For $n \in \mathds{N}_{0}$, the following hold on $\Omega(n)$:
	\begin{enumerate}[label=$\mathrm{(\roman*)}$]
		\item $\log \ell_{H,H'} \leq n$;
		\item $\lvert f_{H,H',n} \rvert = 1$.
	\end{enumerate}
\end{Lemma}

\begin{proof} \let\qed\relax
	\begin{enumerate}[wide=15pt, label=(\roman*)]
		\item By our normalization, $\lvert \ell_{H,H'}(\boldsymbol{\omega}) \rvert = 1$ for $\boldsymbol{\omega} \in \lambda^{-1}(v_{0})$. Then by 
		\ref{Proposition.van-der-Put-transform-of-special-rational-function-evaluated-on-abitrary-arrow},
		$\lVert \ell_{H,H'}(\boldsymbol{\omega}) \rVert_{v} \leq q^{n}$ for $v \in \mathbf{V}(\mathcal{BT})$ whenever $d(v_{0}, v) \leq n$, which gives the assertion.
		\item $\lvert f_{H,H',n}(\boldsymbol{\omega}) \rvert = \lvert 1 + \pi^{n} \ell_{H,H'}(\boldsymbol{\omega}) \rvert \leq 1$ on $\Omega(n)$ by (i), with equality at least if $n=0$ or $\boldsymbol{\omega}$ 
		doesn't belong to $\lambda^{-1}(v)$, where $v$ is a vertex with $d(v_{0}, v) = n$, since in this case $\log \ell_{H,H'}(\boldsymbol{\omega}) < n$. But the equality must also hold for $\boldsymbol{\omega}$ 
		with $\lambda(\boldsymbol{\omega}) =$ such a $v$, due to the linear interpolation property \ref{Theorem.Invertible-holomorphic-functions-are-affine-on-certain-sets} of $\log_{q} \lVert f_{H,H',n} \rVert_{x}$ 
		for $x$ belonging to an arrow $e = (v',v)$ with $d(v_{0}, v') = n-1$. \hfill \mbox{$\square$}
	\end{enumerate}
\end{proof}

\begin{Definition}
	A vertex $v \in \mathbf{V}(\mathcal{BT})$ is called \textbf{$\boldsymbol{n}$-special} ($n \in \mathds{N}_{0}$) if there exists a (necessarily uniquely determined) path $(v_{0}, v_{1}, \dots, v_{n} = v) \in \mathbf{A}_{v_{0},1,n}$,
	i.e., the arrows $e_{i} = (v_{i-1}, v_{i})$, $i = 1,2, \dots, n$ all have type 1, and $d(v_{0}, v) = n$. (By definition, $v_{0}$ is $0$-special.) An arrow $e \in \mathbf{A}(\mathcal{BT})$ is \textbf{$\boldsymbol{n}$-special} $(n \in \mathds{N})$ if
	$o(e)$ is $(n-1)$-special and $t(e)$ is $n$-special, that is, if it appears as some $e_{n}$ as above. Also, the path $(v_{0}, \dots, v_{n}) = (e_{1}, \dots, e_{n})$ is called \textbf{$\boldsymbol{n}$-special}. An arrow $e$ with 
	$d(v_{0}, o(e)) = n$ is \textbf{inbound} (of level $n$) if it belongs to $\mathcal{BT}(n)$, and \textbf{outbound} otherwise. That is, $e$ is inbound $\Leftrightarrow d(v_{0}, t(e)) \leq n$.
\end{Definition}

\subsection{} \label{subsection.Description-of-P(f_H_H'_n)-on-n+1-special-arrows} 
Next, we describe the restriction of $P(f_{H,H',n})$ to $(n+1)$-special arrows $e$. Let $n \in \mathds{N}_{0}$, and choose hyperplanes $H,H'$ of $V$, given as 
$H = \mathrm{ker}(\ell_{\mathbf{y}})$, $H' = \mathrm{ker}(\ell_{\mathbf{y}'})$ as in \eqref{Eq.Definition-of-relative-to-ell-H-H'}, $G = H^{\perp} = K\mathbf{y}$, $G' = K\mathbf{y}'$. 
Assume that $\mathbf{y}$ and $\mathbf{y}'$ are not 1-equivalent \eqref{Eq.n-equivalence-of-vectors}, that is, $\tau_{G}(v_{0}) \neq \tau_{G'}(v_{0})$.

\begin{enumerate}[wide=15pt, label=(\roman*)]
	\item According to Corollary \ref{Corollary.van-der-Put-transform-of-special-rational-function-evaluated-on-arbitrary-arrow-of-type-1}, $\ell_{H,H'} = \ell_{\mathbf{y}}/\ell_{\mathbf{y}'}$ has the property 
	that $\log \ell_{H,H'}$ grows by 1 in each step of the $(n+1)$-special path
	\begin{equation} \label{Eq.special-n+1-special-path}
		(v_{0}, v_{1}, \dots, v_{n}, v_{n+1}) = (e_{1}, e_{2}, \dots, e_{n+1})
	\end{equation}
	from $v_{0}$ toward $G'$. Together with \ref{Lemma.Properties-of-f-H-H'-n-on-Omega-n} (ii), this implies that $P(f_{H,H',n})(e_{n+1}) = 1$.
	\item On the other hand, again by Corollary \ref{Corollary.van-der-Put-transform-of-special-rational-function-evaluated-on-arbitrary-arrow-of-type-1}, $\log \ell_{H,H'} < n$
	on $\lambda^{-1}(v)$ for each $n$-special $v$ different from $v_{n}$. By a variation of the linear interpolation argument in the proof of \ref{Lemma.Properties-of-f-H-H'-n-on-Omega-n} (ii), $P(f_{H,H',n})(e) = 0$ for each $(n+1)$-special
	arrow $e$ with $o(e) \neq v_{n}$.
	\item The function $u \defeq f_{H,H',n} = (\ell_{\mathbf{y}'} + \pi^{n}\ell_{\mathbf{y}})/\ell_{\mathbf{y}'}$ satisfies $\lVert u \rVert_{v_{n}} = 1$. Its reduction $\bar{u}$ as a rational function on the reduction
	\begin{equation} \label{Eq.Preimage-of-special-vector-under-affinoid-lambda}
		\bar{\Omega}_{v_{n}} = (\mathds{P}(L^{\wedge}/\pi L^{\wedge})/\mathds{F}) \setminus \bigcup \bar{H} \qquad \text{(see \eqref{Eq.Description-for-Omega-bar-v}, here $v_{n} = [L]$)}
	\end{equation}
	of $\Omega_{v_{n}} = \lambda^{-1}(v_{n})$ has a simple pole along the hyperplane $\bar{H}_{e_{n+1}}$ corresponding to the arrow $e_{n+1}$, a simple zero along a unique $\bar{H}_{e}$, where 
	$e = (v_{n}, w)$, and neither zeroes nor poles along the other hyperplanes that appear in \eqref{Eq.Preimage-of-special-vector-under-affinoid-lambda}. The hyperplane $\bar{H}_{e}$ is the vanishing 
	locus in $\mathds{P}(L^{\wedge}/\pi L^{\wedge})/\mathds{F}$ of the
	reduction of the form $\ell_{\mathbf{y}'} + \pi^{n} \ell_{\mathbf{y}} = \ell_{\mathbf{y}''}$; accordingly, $w = \tau_{G''}(v_{n})$, where $G'' = K\mathbf{y}''$ and
	\begin{equation} \label{Eq.Description-for-y''}
		\mathbf{y}'' = \mathbf{y}' + \pi^{n} \mathbf{y}.
	\end{equation}
	\item If $\mathbf{y}'$ is fixed and $\mathbf{y}$ runs through the elements of $L_{0} \setminus \pi L_{0}$ not 1-equivalent with $\mathbf{y}'$, then the corresponding $\mathbf{y}''$ are $n$-equivalent but 
	not $(n+1)$-equivalent with $\mathbf{y}'$ (cf. \eqref{Eq.n-equivalence-of-vectors}). In this way we get all the $(n+1)$-classes with this property, that is, all the $(n+1)$-special paths 
	$(e_{1}, e_{2}, \dots, e_{n}, e)$ which agree with the path $(e_{1}, \dots, e_{n}, e_{n+1})$ of \eqref{Eq.special-n+1-special-path} except for the last arrow. We collect what has been shown.
\end{enumerate}

\begin{Proposition} \label{Proposition.Hyperplanes}
	\begin{enumerate}[label=$\mathrm{(\roman*)}$]
		\item Let $H, H'$ be two hyperplanes in $V$, $G = H^{\perp}$, $G' = H'^{\perp}$, with $\tau_{G}(v_{0}) \neq \tau_{G'}(v_{0})$ and $n \in \mathds{N}_{0}$. Put $v_{i} \vcentcolon = (\tau_{G'})^{i}(v_{0})$. If $e$ is an 
		$(n+1)$-special arrow then
		\par\noindent\begin{minipage}{\linewidth}		
		\begin{equation}
			P(f_{H,H',n})(e) = \begin{cases} +1, &\text{if } e=(v_{n}, v_{n+1}), \\ -1, &\text{if }e =(v_{n}, w), \\ \hphantom{+}0, &\text{otherwise.} \end{cases}
		\end{equation}
		\vspace*{2pt}
		\end{minipage}
		Here $w = \tau_{G''}(v_{n}) \neq v_{n+1}$, where $G''= K\mathbf{y}''$ with $\mathbf{y}'' = \mathbf{y}' + \pi^{n} \mathbf{y}$
		as described in \ref{subsection.Description-of-P(f_H_H'_n)-on-n+1-special-arrows}, notably in \eqref{Eq.Description-for-y''}.
		\item If $H'$ is fixed, each $(n+1)$-special arrow $e \neq (v_{n}, v_{n+1})$ with $o(e) = v_{n}$ occurs through a suitable choice of $H$ as the arrow $e = (v,w)$ where $P(f_{H,H',n})$ evaluates to $-1$. 
		\hfill \mbox{$\square$}
	\end{enumerate}
\end{Proposition}

The next result, technical in nature, is crucial for the proof of Theorem \ref{Theorem.van-der-Put-sequence-short-exact-sequence-of-GK-modules}. Its proof is postponed to the next section.

\begin{Proposition} \label{Proposition.Crucial-technical-tool-for-main-theorem}
	Let $n \in \mathds{N}_{0}$ and $\varphi \in \mathbf{H}(\mathcal{BT}, \mathds{Z})$ be such that $\varphi(e) = 0$ for arrows $e$ that either belong to $\mathcal{BT}(n)$ or are $(n+1)$-special. Then 
	$\varphi(e) = 0$ for all arrows $e$ of $\mathcal{BT}(n+1)$.
\end{Proposition}

Now we are able to show (modulo \ref{Proposition.Crucial-technical-tool-for-main-theorem}) the principal result.

\begin{Theorem} \label{Theorem.van-der-Put-sequence-short-exact-sequence-of-GK-modules}
	The van der Put map $P \colon \mathcal{O}(\Omega)^{*} \to \mathbf{H}(\mathcal{BT}, \mathds{Z})$ is surjective, and so the sequence
	\begin{equation}
		\begin{tikzcd}
			1 \ar[r] & C^{*} \ar[r] &\mathcal{O}(\Omega)^{*} \ar[r] & \mathbf{H}(\mathcal{BT}, \mathds{Z}) \ar[r] &0 \tag{0.2}
		\end{tikzcd}
	\end{equation}
	is a short exact sequence of $G(K)$-modules.
\end{Theorem}	

\begin{proof}\let\qed\relax
	\begin{enumerate}[wide=15pt, label=(\roman*)]
		\item Let $\varphi \in \mathbf{H}(\mathcal{BT}, \mathds{Z})$ be given. By successively subtracting $P(u_{n})$ from $\varphi$, where $(u_{n})_{n \in \mathds{N}}$ is a suitable series of functions in $\mathcal{O}(\Omega)^{*}$ with
		$u_{n} \to 1$ locally uniformly (i.e., uniformly on affinoids) we will achieve that
		\[
			\varphi - P\Big( \prod_{1 \leq i \leq n} u_{i} \Big) \equiv 0 \qquad \text{on $\mathcal{BT}(n)$}.
		\] 
		Then $\varphi = P(u)$, where $u = \lim_{n \to \infty} \prod_{1 \leq i \leq n} u_{i}$ is the limit function.
		\item From condition $(\mathrm{B}_{1})$ for $\varphi$ and Proposition \ref{Proposition.van-der-Put-transform-on-special-rational-function-evaluated-on-special-arrow-of-type-1} we find a function $u_{1}$, namely a suitable finite
		product of functions of type $\ell_{H,H'}$, such that $(\varphi - P(u_{1}))(e) = 0$ for each $e \in \mathbf{A}_{v_{0}, 1}$. By condition (C), $\varphi - P(u_{1})$ vanishes on all $e \in \mathbf{A}_{v_{0}}$, and thus by (A) on all $e$ that
		belong to $\mathcal{BT}(1) = \mathrm{st}(v_{0})$.
		\item Suppose that $u_{1}, \dots, u_{n} \in \mathcal{O}(\Omega)^{*}$ are constructed $(n \in \mathds{N})$ such that for $1 \leq i \leq n$
		\begin{enumerate}[label=(\alph*), itemindent=15pt]
			\item $P(u_{i}) \equiv 0$ on $\mathcal{BT}(i-1)$,
			\item $u_{i} \equiv 1 \pmod{\pi^{[(i-1)/2]}}$ on $\mathcal{BT}([(i-1)/2])$; here $[\cdot]$ is the Gauß bracket;
			\item $\varphi - P(\prod_{1 \leq i \leq n} u_{i}) \equiv 0$ on $\mathcal{BT}(n)$
		\end{enumerate}
		hold. (Condition (a) is empty for $i=1$ and therefore trivially fulfilled.) We are going to construct $u_{n+1}$ such that $u_{1}, \dots, u_{n+1}$ fulfill the conditions on level $n+1$.
		\item From (c) and $(\mathrm{B}_{1})$ we have for $n$-special vertices $v$ and $\psi \vcentcolon = \varphi - P(\prod_{1 \leq i \leq n} u_{i}) \in \mathbf{H}(\mathcal{BT}, \mathds{Z})$:
		\[
			\sum_{e \in \mathbf{A}_{v,1} \text{ outbound}} \psi(e) = \sum_{e \in \mathbf{A}_{v,1}} \psi(e) = 0.
		\]
		\item According to Proposition \ref{Proposition.Hyperplanes}, we find $u_{n+1}$, viz, a suitable product of functions $f_{H,H',n}$, such that
		\[
			\big(\psi - P(u_{n+1})\big)(e) = \bigg( \varphi - P\Big(\prod_{1 \leq i \leq n+1} u_{i} \Big)\bigg)(e) = 0
		\]
		on all $(n+1)$-special arrows $e$. Furthermore, that $u_{n+1}$ (like the functions $f_{H,H',n}$, see Lemma \ref{Lemma.Properties-of-f-H-H'-n-on-Omega-n} (ii)) satisfies
		$P(u_{n+1}) \equiv 0$ on $\mathcal{BT}(n)$, i.e., condition (a), and condition (b): $u_{n+1} \equiv 1 \pmod{\pi^{[n/2]}}$ on $\mathcal{BT}([n/2])$. Hence $\varphi - P(\prod_{1 \leq i \leq n+1} u_{i})$ vanishes on arrows which belong to
		$\mathcal{BT}(n)$ or are $(n+1)$-special. Using Proposition \ref{Proposition.Crucial-technical-tool-for-main-theorem}, we see that 
		$\varphi - P(\prod_{1 \leq i \leq n+1} u_{i})$ vanishes on $\mathcal{BT}(n+1)$. That is, conditions (a), (b), (c) hold
		for $u_{1}, \dots, u_{n+1}$, and we have inductively constructed an infinite series $u_{1}, u_{2}, \dots$ with (a), (b) and (c) for all $n$.
		\item It follows from (b) that the infinite product
		\[
			u = \prod_{i \in \mathds{N}} u_{i}
		\]
		is normally convergent on each $\Omega(n)$ and thus defines a holomorphic invertible function $u$ on $\Omega$. Its van der Put transform $P(u)$ restricted to $\mathcal{BT}(n)$ depends only on $u_{1}, \dots, u_{n}$, due to (c), and
		thus agrees with $\varphi$ reduced to $\mathcal{BT}(n)$. Therefore, $\varphi = P(u)$, and the result is shown.
		\hfill \mbox{$\square$}
	\end{enumerate}
\end{proof}

\section{The group $\mathbf{H}(\mathcal{BT}, \mathds{Z})$}

\subsection{} We start with the 
\begin{proof}[Proof of Proposition \ref{Proposition.Crucial-technical-tool-for-main-theorem}]
	\begin{enumerate}[wide=15pt, label=(\roman*)]
		\item The requirements of Proposition \ref{Proposition.Crucial-technical-tool-for-main-theorem} for $\varphi \in \mathbf{H}(\mathcal{BT}, \mathds{Z})$ on level $n \in \mathds{N}_{0}$ will be labelled by $\mathrm{R}(n)$.
		\item Suppose that $\mathrm{R}(n)$ holds for $\varphi$. Then $\varphi$ vanishes on all arrows $\mathbf{A}_{v,1}$ whenever $v$ is $n$-special, since such an $e$ is either $(n+1)$-special or belongs to $\mathcal{BT}(n)$. Hence by
		conditions (C) and (A) of \ref{subsection.Conditions-for-varphi-in-A(BT)}, $\varphi(e) = 0$ whenever $e$ is contiguous with $v$, i.e., if $e$ belongs to $\mathrm{st}(v)$. This shows, in particular, that Proposition 
		\ref{Proposition.Crucial-technical-tool-for-main-theorem} holds for $n=0$.
		\item Let $v \in \mathbf{V}(\mathcal{BT})$ have distance $d(v_{0},v) = n$, but be not necessarily $n$-special. For the same reason as in (ii), $\varphi$ vanishes identically on $\mathrm{st}(v)$ if it vanishes on all outbound arrows
		$e \in \mathrm{A}_{v,1}$. Hence it suffices to show
		\begin{equation} \label{Eq.Assertion-for-varphi-in-the-proof-of-Proposition-3-9}
			\varphi(e) = 0 \quad \text{for outbound arrows $e$ of type 1 and level $n$}. \tag{O}
		\end{equation}
		\item For a vertex $v$ with $d(v_{0}, v) = n$, we let $s(v)$ be the distance to the next $w \in \mathbf{V}(\mathcal{BT})$ which is $n$-special. We are going to show assertion 
		\eqref{Eq.Assertion-for-varphi-in-the-proof-of-Proposition-3-9} by induction on $s(o(e))$.
		\item By $\mathrm{R}(n)$, \eqref{Eq.Assertion-for-varphi-in-the-proof-of-Proposition-3-9} holds if $s = s(o(e)) = 0$, i.e., if $o(e)$ is $n$-special. Therefore, suppose that $s > 0$. By the preceding we are 
		reduced to showing 
		\begin{equation} \label{Eq.Reduced-assertion-for-varphi-in-the-proof-of-Proposition-3-9}
			\begin{array}{l}
				\text{Let $e$ be an outbound arrow of type 1, level $n$, and with $s = s(o(e)) > 0$.} \\
				\text{Then $e$ belongs to $\mathrm{st}(\tilde{v})$, where $d(v_{0}, \tilde{v}) = n$ and $s(\tilde{v}) < s$.}
			\end{array} \tag{P}
		\end{equation}
		\item We reformulate \eqref{Eq.Reduced-assertion-for-varphi-in-the-proof-of-Proposition-3-9} in lattice terms. Representing $v_{0} = [L_{0}]$ through $L_{0} = O^{r}$, the vertices 
		$v \in \mathbf{V}(\mathcal{BT})$ correspond one-to-one to sublattices $L$ of full rank $r$ which satisfy  $L \subset L_{0}$, $L \not\subset \pi L_{0}$. For such a vertex $v$ or its lattice $L$, we let 
		$(n_{1}, n_{2}, \dots, n_{r})$ with $n_{1} \geq n_{2} \geq \dots \geq n_{r} = 0$ be the sequence of elementary divisors (\textbf{sed}) of $L_{0}/L$ ($n_{r} = 0$ as $L \not\subset \pi L_{0}$). That is,
		\[
			L_{0}/L \cong O/(\pi^{n_{1}}) \times \cdots \times O/(\pi^{n_{r}}).
		\]
		Then $n_{1} = d(v_{0}, v)$, and $v$ is $n$-special if and only if its sed is $(n, \dots,n, 0)$.
		\item Let $e = (v,v')$ be given as required for \eqref{Eq.Reduced-assertion-for-varphi-in-the-proof-of-Proposition-3-9}, $v=[L]$, $v' = [L']$, where $\pi^{n+1}L_{0} \subset L' \subset L \subset L_{0}$. Let
		$(n_{1} = n, n_{2}, \dots, n_{r})$ be the sed of $L_{0}/L$. Then, as $\dim_{\mathds{F}}(L/L') = r-1$ and $d(v_{0}, v') = n+1$, $(n_{1}' = n+1, \dots, n_{r-1}' = n_{r-1}+1, n_{r})$ is the sed of $L_{0}/L'$. This means 
		that $L_{0}$ has an ordered $O$-basis $\{x_{1}, \dots, x_{r}\}$ such that $\{ \pi^{n+1}x_{1}, \pi^{n_{2}+1}x_{2}, \dots, \pi^{n_{r-1}+1}x_{r-1}, \pi^{n_{r}}x_{r}\}$ is a basis of $L'$ and 
		$\{ \pi^{n} x_{1}, \pi^{n_{2}} x_{2}, \dots, \pi^{n_{r}} x_{r}\}$ is a basis of $L$. Assume that $k$ with $1 \leq k \leq r-1$ is minimal with $n_{r-1} = n_{k}$. Let $M$ be the sublattice of $L_{0}$ with basis 
		$\{ \pi^{n}x_{1}, \dots, \pi^{n}x_{r-1}, x_{r}\}$. Then $w = [M]$ is $n$-special and $s(v) = d(v,w) = n - n_{r-1}$, which by assumption is positive.
		Put $\tilde{L}$ for the lattice with basis 
		\[
			\{ \pi^{n}x_{1}, \pi^{n_{2}}x_{2}, \dots, \pi^{n_{k-1}}x_{k-1}, \pi^{n_{k}+1}x_{k}, \pi^{n_{k+1}+1}x_{k+1}, \dots, \pi^{n_{r-1}+1}x_{r-1}, \pi^{n_{r}}x_{r}\}.
		\] 
		The vertex $\tilde{v} \vcentcolon = [\tilde{L}]$ satisfies
		\begin{multline}
			d(v_{0}, \tilde{v}) = n, ~d(v,\tilde{v}) = 1 = d(v', \tilde{v}) \\ \text{and} \quad s(\tilde{v}) = d(w,\tilde{v}) = n - n_{r-1}-1 = s(v)-1.
		\end{multline}
		Hence $e = (v,v')$ belongs to $\mathrm{st}(\tilde{v})$, where $\tilde{v}$ is as wanted for assertion \eqref{Eq.Reduced-assertion-for-varphi-in-the-proof-of-Proposition-3-9}.
	\end{enumerate}
	This finishes the proof of Proposition \ref{Proposition.Crucial-technical-tool-for-main-theorem}.
\end{proof}

\begin{Corollary} \label{Corollary.Statement-for-varphi-who-is-0-on-i-special-arrows}
	Let $\varphi \in \mathbf{H}(\mathcal{BT}, \mathds{Z})$ be such that $\varphi(e) = 0$ for all $i$-special arrows $e$, where $1 \leq i \leq n$. Then $\varphi \equiv 0$ on $\mathcal{BT}(n)$.
\end{Corollary}

\begin{proof}
	This follows by induction from \ref{Proposition.Crucial-technical-tool-for-main-theorem}.
\end{proof}

\subsection{} Next we give a different description of $\mathbf{H}(\mathcal{BT}, \mathds{Z})$, see Theorem \ref{Theorem.Abstraction-to-arbitrary-abelian-groups}. Let $v$ be an $n$-special vertex $(n \geq 1)$, $v^{*}$ its predecessor on the uniquely determined $n$-special path $(v_{0}, v_{1}, \dots, v_{n-1} = v^{*}, v)$ from $v_{0}$ to $v$, and $e^{*}$ the $n$-special arrow $(v^{*}, v)$.
Its inverse $\bar{e^{*}} = (v,v^{*})$ belongs to $\mathbf{A}_{v,r-1}$.

\begin{Lemma} \label{Lemma.If-and-only-if-condition-for-inboundness}
	In the given situation, $e \in \mathbf{A}_{v,1}$ is inbound if and only if $e \prec \bar{e^{*}}$.
\end{Lemma}

\begin{proof}
	As the stabilizer $\mathrm{GL}(r,O)$ of $L_{0} = O^{r}$ acts transitively on $n$-special vertices or arrows, we may suppose that $v = [L_{n}]$, where $L_{n}$ is the $O$-lattice with basis 
	$\{\pi^{n}x_{1}, \dots, \pi^{n}x_{r-1}, x_{r} \}$, and thus $v^{*} = [L_{n-1}]$. (Here $\{x_{1}, \dots, x_{r}\}$ is the standard basis of $L_{0}$.) Under \eqref{Eq.Canonical-bijection-of-Avt-and-Gr-FtLpiL}, 
	$\bar{e^{*}}$ corresponds to the $(r-1)$-dimensional subspace $\pi L_{n-1}/\pi L_{n}$ of the $r$-dimensional $\mathds{F}$-space $L_{n}/\pi L_{n}$, which has 
	the $(\overline{\pi^{n}x_{i}}) = \pi^{n}x_{i} \pmod{\pi L_{n}}$, $1 \leq i < r$, as a basis. Let $\bar{G}$ be a line in $L_{n}/\pi L_{n}$ with pre-image $G$ in 
	$L_{n}$, and let $e_{\bar{G}} = (v, v_{\bar{G}})$ be the arrow of type 1 determined by $\bar{G}$. Then $v_{\bar{G}} = [G]$ and
	\[
		e_{\bar{G}} \prec e^{*} \Longleftrightarrow \bar{G} \subsetneq \pi L_{n-1}/\pi L_{n} \Longleftrightarrow G \subset \pi L_{n-1}.
	\]
	If this is the case, $\pi^{n}L_{0} \subset L_{n} \subset \pi^{-1}G \subset L_{n-1} \subset L_{0}$, that is, $d(v_{0}, [G]) \leq n$, and $e_{\bar{G}}$ is inbound. On the other hand, if $G \not\subset \pi L_{n-1}$,
	then $\pi^{-1}G \not\subset L_{0}$. Since $\pi^{n}L_{0} \not\subset G$, we then have $d(v_{0}, [G])=n+1$, and $e_{\bar{G}}$ is outbound.
\end{proof}

\subsection{} \label{subsection.Reformulation-of-B-1-at-n-special-v} We may now reformulate condition $(\mathrm{B}_{1})$ for $\varphi \in \mathbf{H}(\mathcal{BT}, \mathds{Z})$ at the $n$-special vertex $v$ of level $n \geq 1$ 
as follows: Splitting
\begin{equation}
	\mathbf{A}_{v,1} = \mathbf{A}_{v,1,\mathrm{in}} \cupdot \mathbf{A}_{v,1,\mathrm{out}}
\end{equation}
into the subsets of inbound / outbound arrows (note that $e \in \mathbf{A}_{v,1}$ is outbound if and only if it is $(n+1)$-special), $(\mathrm{B}_{1})$ reads
\begin{equation*}
	0 = \sum_{e \in \mathbf{A}_{v,1}} \varphi(e) = \sum_{e \in \mathbf{A}_{v,1,\mathrm{in}}} \varphi(e) + \sum_{e \in \mathbf{A}_{v,1,\mathrm{out}}} \varphi(e) = \varphi(\bar{e^{*}}) + \sum_{e \in \mathbf{A}_{v,1,\mathrm{out}}} \varphi(e)
\end{equation*}
(where we used \ref{Lemma.If-and-only-if-condition-for-inboundness} and condition (C) for $\varphi(\bar{e^{*}})$), i.e., as the flow condition
\begin{equation} \label{Eq.Flow-condition}
	\varphi(e^{*}) = \sum_{e \in \mathbf{A}_{v,1,\mathrm{out}}} \varphi(e).
\end{equation}
The number of terms in the sum is
\begin{equation} \label{Eq.Cardinality-of-outbound-arrows-of-type-1}
	\# \mathbf{A}_{v,1,\mathrm{out}} = \# \mathbf{A}_{v,1} - \# \mathbf{A}_{v,1,\mathrm{in}} = \# \mathds{P}^{r-1}(\mathds{F}) - \# \mathds{P}^{r-2}(\mathds{F}) = q^{r-1}.
\end{equation}

\subsection{} Let $\mathcal{T}_{v_{0}}$ be the full subcomplex of $\mathcal{BT}$ composed of the $n$-special vertices ($n \in \mathds{N}_{0}$) along with the 1-simplices connecting them. In other words,
$\mathcal{T}_{v_{0}}$ is the union of the paths $\mathbf{A}_{v_{0},1,n}$, where $n \in \mathds{N}$, see \ref{subsubsection.Identification-of-arrows-with-special-Grassmannians}.
It is connected, one-dimensional and cycle-free, hence a tree. The valence (= number of neighbors) of $v_{0}$ is $\# \mathds{P}^{r-1}(\mathds{F}) = (q^{r}-1)/(q-1)$, the valence of each other vertex $v \neq v_{0}$ is $q^{r-1}+1$, as we
read off from \eqref{Eq.Cardinality-of-outbound-arrows-of-type-1}. Let further $\mathcal{T}_{v_{0}}(n) \vcentcolon = \mathcal{T}_{v_{0}} \cap \mathcal{BT}(n)$.

\subsubsection{} We define $\mathbf{H}(n)$ as the image of $\mathbf{H}(\mathcal{BT}, \mathds{Z})$ in $\{ \varphi \colon \mathbf{A}(\mathcal{BT}(n)) \to \mathds{Z}\}$ obtained by restriction. Hence \stepcounter{equation}
\begin{equation} \label{Eq.H(BT,Z)-as-projective-limit}
	\mathbf{H}(\mathcal{BT},\mathds{Z}) = \varprojlim_{n \in \mathds{N}} \mathbf{H}(n).
\end{equation}
Put further
\begin{equation}
	\mathbf{H}'(n) \vcentcolon = \left\{ \varphi \colon \mathbf{A}(\mathcal{T}_{v_{0}}(n)) \longrightarrow \mathds{Z} \, \middle\vert  \begin{array}{l} \varphi \text{ is subject to \eqref{Eq.Condition-1-H'n} and \ref{Eq.Condition-2-H'n}} \\ \text{for each $i$-special $v$, $0 \leq i < n$} \end{array} \!\!\right\}.
\end{equation}
Here $\mathbf{A}(\mathcal{S})$ is the set of arrows (oriented 1-simplices) of the simplicial complex $\mathcal{S}$, and the conditions are 
\begin{equation} \label{Eq.Condition-1-H'n}
	\varphi(e) + \varphi(\bar{e}) = 0 \quad \text{for each arrow $e$ with inverse $\bar{e}$};
\end{equation}	
\begin{equation} \label{Eq.Condition-2-H'n}
	\sum_{\substack{ e \in \mathbf{A}(\mathcal{T}_{v_{0}}) \\ o(e) = v }} \varphi(e) = 0. \tag*{(4.6.5)($v$)}
\end{equation}

\subsection{} Equality \eqref{Eq.Flow-condition} together with the condition $(\mathrm{B}_{1})$ at $v_{0}$ states that the restriction of $\varphi \in \mathbf{H}(\mathcal{BT}, \mathds{Z})$ to
$\mathcal{T}_{v_{0}}(n)$ is an element of $\mathbf{H}'(n)$. Therefore, restriction defines homomorphisms $r_{n} \colon \mathbf{H}(n) \to \mathbf{H}'(n)$, which make the diagram (with natural maps
$q_{n}$, $q_{n}'$)
\begin{equation}
	\begin{tikzcd}
		\mathbf{H}(n+1) \ar[r, "r_{n+1}"]	\ar[d, "q_{n}"]	& \mathbf{H}'(n+1) \ar[d, "q_{n}'"] \\
		\mathbf{H}(n) \ar[r, "r_{n}"]									& \mathbf{H}'(n)
	\end{tikzcd}
\end{equation}
commutative. Note that both $q_{n}$ and $q_{n}'$ are surjective, the first by definition, the second one since $\mathcal{T}_{v_{0}}$ is a tree. 
Corollary \ref{Corollary.Statement-for-varphi-who-is-0-on-i-special-arrows} may be rephrased as

\begin{Proposition} \label{Proposition.rn-is-injective}
	$r_{n}$ is injective for $n \in \mathds{N}$. \hfill \mbox{$\square$}
\end{Proposition}

\begin{Lemma} \label{Lemma.rn-is-surjective}
	$r_{n}$ is also surjective.
\end{Lemma}

\begin{proof}
	For $n=1$, this is implicit in the proof of Theorem \ref{Theorem.van-der-Put-sequence-short-exact-sequence-of-GK-modules} (i.e., one may arbitrarily prescribe the value of $\varphi \in \mathbf{H}(\mathcal{BT}, \mathds{Z})$
	on $e \in \mathbf{A}_{v_{0},1}$, subject only to $(\mathrm{B}_{1})$ at $v_{0}$).
	
	For $n \geq 1$, let $Q_{n+1}$ (respectively $Q_{n+1}'$) be the kernel of $q_{n}$ (resp. $q_{n}'$). Then $r_{n+1}(Q_{n+1}) \subset Q_{n+1}'$, and we have the commutative diagram with exact rows
	\[
		\begin{tikzcd}
			0 \ar[r]	&Q_{n+1} \ar[r]	\ar[d]	&\mathbf{H}(n+1)	\ar[r] \ar[d, "r_{n+1}"]	& \mathbf{H}(n) \ar[r] \ar[d, "r_{n}"]	& 0 \\
			0 \ar[r]	& Q_{n+1}' \ar[r]			&\mathbf{H}'(n+1) \ar[r]			& \mathbf{H}'(n) \ar[r]			& 0.
		\end{tikzcd}
	\]
	By induction hypothesis, $r_{n}$ is surjective, so the surjectivity of $r_{n+1}$ is implied by
	\begin{equation} \label{Eq.Reduced-assertion-by-induction}
		r_{n+1}(Q_{n+1}) = Q_{n+1}'. \tag{$\ast$}
	\end{equation}
	But
	\begin{align*}
		Q_{n+1}		&= \{ \varphi \in \mathbf{H}(n+1) 	\mid \varphi \equiv 0 \text{ on } \mathcal{BT}(n) \} \\
		Q_{n+1}'	&= \{ \varphi \in \mathbf{H}'(n+1) 	\mid \varphi \equiv 0 \text{ on } \mathcal{T}_{v_{0}}(n) \},
	\end{align*}
	so \eqref{Eq.Reduced-assertion-by-induction} follows from the existence of sufficiently many elements of $Q_{n+1}$ (e.g., the classes in $\mathbf{H}(n+1)$ of the $P(f_{H,H',n})$) which have sufficiently independent values on the arrows
	in $\mathcal{T}_{v_{0}}(n+1)$ not in $\mathcal{T}_{v_{0}}(n)$. See also the proof of Theorem \ref{Theorem.van-der-Put-sequence-short-exact-sequence-of-GK-modules}, steps (iv) and (v).
\end{proof}

\subsection{} Let $\mathbf{H}(\mathcal{T}_{v_{0}}, \mathds{Z}) = \varprojlim_{n \in \mathds{N}} \mathbf{H}'(n)$ be the group of functions $\varphi \colon \mathbf{A}(\mathcal{T}_{v_{0}}) \to \mathds{Z}$ which satisfy \eqref{Eq.Condition-1-H'n} 
and \ref{Eq.Condition-2-H'n} for all vertices $v$ of $\mathcal{T}_{v_{0}}$. Similarly, we define $\mathbf{H}(\mathcal{T}_{v_{0}}, A)$ for an arbitrary abelian group $A$ instead of $\mathds{Z}$. That is, elements of 
$\mathbf{H}(\mathcal{T}_{v_{0}}, A)$ are characterized by conditions analogous with (A) and (B) of \ref{subsection.Conditions-for-varphi-in-A(BT)}, while (C) is not applicable. Putting together the
considerations of \ref{subsection.Reformulation-of-B-1-at-n-special-v} with \ref{Proposition.rn-is-injective} and \ref{Lemma.rn-is-surjective}, we find
\begin{equation} \label{Eq.Canonical-isomorphism-between-H(BT,Z)-and-H(TV0,Z)}
	\begin{tikzcd}
		\mathbf{H}(\mathcal{BT}, \mathds{Z}) \ar[r, "\cong"] &\mathbf{H}(\mathcal{T}_{v_{0}}, \mathds{Z}), \tag{4.11}
	\end{tikzcd}
\end{equation}
where the canonical isomorphism is given by restricting $\varphi \in \mathbf{H}(\mathcal{BT}, \mathds{Z})$, $\varphi \colon \mathbf{A}(\mathcal{BT}) \to \mathds{Z}$ to the subset $\mathbf{A}(\mathcal{T}_{v_{0}})$ of $\mathbf{A}(\mathcal{BT})$.

In what follows, $A$ is an arbitrary abelian group. The next result is a consequence of the above.
\stepcounter{subsection}

\begin{Proposition} \label{Proposition.Canonical-maps}
	Restriction to the arrows of $\mathcal{T}_{v_{0}}$ yields an isomorphism
	\begin{equation} \label{Eq.Canonical-isomorphism-between-H(BT,A)-and-H(TV0,A)}
		\begin{tikzcd}
			\mathbf{H}(\mathcal{BT}, A) \ar[r, "\cong"]	& \mathbf{H}(\mathcal{T}_{v_{0}}, A).
		\end{tikzcd}
	\end{equation}
\end{Proposition}

\begin{proof}
	It suffices to observe that the preceding results Proposition 3.10, Corollary 4.2, Proposition 4.8, and Lemma 4.9 remain valid --- with identical proofs --- for $A$-valued functions instead
	of $\mathds{Z}$-valued functions.
\end{proof}

\subsection{} Recall that an $A$-valued distribution on a compact totally disconnected topological space $X$ is a a map $\delta \colon U \mapsto \delta(U) \in A$ from the set of compact-open subspaces 
$U$ of $X$ to $A$ which is additive in finite disjoint unions. We call $\delta(U)$ the \textbf{volume} of $U$ \textbf{with respect to} $\delta$. The \textbf{total mass} (or volume) of $\delta$ is $\delta(X)$.

We apply this to the situation (see \ref{subsubsection.Identification-of-arrows-with-special-Grassmannians} - \eqref{Eq.Special-grassmannian-is-compact-and-open-in-Gr-KtV}) where
\begin{equation}
	X = \mathrm{Gr}_{K,1}(V) = \left\{ \begin{array}{l} \text{lines $G$ of} \\ \text{the $K$-space $V$} \end{array} \right\} = \mathds{P}(V).
\end{equation}
As we have identified $V$ with its dual $V^{\wedge}$ through the bilinear form $\langle \cdot, \cdot \rangle$, we also have an identification 
\[
	\begin{tikzcd}
		\mathrm{Gr}_{K,1}(V) = \mathds{P}(V) \ar[r, "{\cong}"]	& \mathds{P}(V^{\wedge}) = \mathrm{Gr}_{K,r-1}(V)
	\end{tikzcd}
\]
given by $G \mapsto G^{\perp}$. Hence we could state the following assertions concerning distributions on $\mathds{P}(V)$ for distributions on $\mathds{P}(V^{\wedge})$.
\stepcounter{subsubsection}
\subsubsection{} Let $\mathbf{D}(\mathds{P}(V), A)$ be the group of $A$-valued distributions on $\mathds{P}(V)$ with subgroup $\mathbf{D}^{0}(\mathds{P}(V), A)$
of distributions with total mass 0. By \eqref{Eq.Special-grassmannian-is-compact-and-open-in-Gr-KtV}, the sets $\mathds{P}(V)(e) = \mathrm{Gr}_{K,1}(e)$, where $e$ runs through the outbound arrows of
$\mathbf{A}_{v_{0},1,n}$ ($n \in \mathds{N}$), i.e., through the set \stepcounter{equation}
\begin{equation}
	\mathbf{A}^{+}(\mathcal{T}_{v_{0}}) = \{e \in \mathbf{A}(\mathcal{T}_{v_{0}}) \mid e \text{ oriented away from $v_{0}$}\},
\end{equation}
form a basis for the topology on $\mathds{P}(V)$. Therefore, an element $\delta$ of $\mathbf{D}(\mathds{P}(V), A)$ is an assignment
\[
	\delta \colon \mathbf{A}^{+}(\mathcal{T}_{v_{0}}) \longrightarrow A
\]
(where we interpret $\delta(e)$ as the volume of $\mathds{P}(V)(e)$ with respect to $\delta$) subject to the requirement
\begin{equation}
	\delta(e^{*}) = \sum_{\substack{e \in \mathbf{A}^{+}(\mathcal{T}_{v_{0}}) \\ o(e) = t(e^{*})}} \delta(e)
\end{equation}
for each $e^{*} \in \mathbf{A}^{+}(\mathcal{T}_{v_{0}})$. The total mass of $\delta$ is
\begin{equation}
	\delta(\mathds{P}(V)) = \sum_{\substack{e \in \mathbf{A}^{+}(\mathcal{T}_{v_{0}}) \\ o(e) = v_{0}}} \delta(e) = \sum_{e \in \mathbf{A}_{v_{0},1}} \delta(e).
\end{equation}
In view of \eqref{Eq.Flow-condition} and (4.6.5)($v_{0}$), we find that
\begin{equation} \label{Eq.Isomorphism-of-D0(PVDach,A)-and-H(Tv0,A)}
	\begin{tikzcd}
		\mathbf{D}^{0}(\mathds{P}(V), A) \ar[r, "\cong"] &\mathbf{H}(\mathcal{T}_{v_{0}}, A), \tag{4.14}
	\end{tikzcd}
\end{equation}
where some $\delta \colon \mathbf{A}^{+}(\mathcal{T}_{v_{0}}) \to A$ in the left hand side is completed to a map on $\mathbf{A}(\mathcal{T}_{v_{0}})$ by \eqref{Eq.Condition-1-H'n}, i.e., by $\varphi(\bar{e}) = -\varphi(e)$. \stepcounter{subsection}

While both isomorphisms in \eqref{Eq.Canonical-isomorphism-between-H(BT,Z)-and-H(TV0,Z)} (or \eqref{Eq.Canonical-isomorphism-between-H(BT,A)-and-H(TV0,A)}) and \eqref{Eq.Isomorphism-of-D0(PVDach,A)-and-H(Tv0,A)} fail to be 
$G(K)$-e\-qui\-vari\-ant (as $G(K)$ fixes neither $v_{0}$ nor $\mathcal{T}_{v_{0}}$), the resulting isomorphism
\begin{align}
	\mathbf{H}(\mathcal{BT}, \mathbf{A})	&\overset{\cong}{\longrightarrow} \mathbf{D}^{0}(\mathds{P}(V), A) \tag{4.15} \\
											\varphi			&\longmapsto \tilde{\varphi} \nonumber
\end{align}
is. Here the distribution $\tilde{\varphi}$ evaluates on $\mathds{P}(V)(e)$ as $\varphi(e)$ whenever $e$ is an arrow of $\mathcal{BT}$ of type 1 and $\mathds{P}(V)(e)$ is the compact-open subset of lines $G$ 
of $V$ such that $e$ points to $G$. \stepcounter{subsection}

We summarize what has been shown.

\begin{Theorem} \label{Theorem.Abstraction-to-arbitrary-abelian-groups}
	Let $A$ be an arbitrary abelian group. Restricting the evaluation of $\varphi \in \mathbf{H}(\mathcal{BT}, A)$ to arrows of $\mathcal{T}_{v_{0}}$ (resp. arrows of type 1 of $\mathcal{BT}$) yields canonical isomorphisms
	\begin{align*}
		\mathbf{H}(\mathcal{BT}, A) 	&\overset{\cong}{\longrightarrow} \mathbf{H}(\mathcal{T}_{v_{0}}, A) 
	\intertext{resp.}
		\mathbf{H}(\mathcal{BT}, A)	&\overset{\cong}{\longrightarrow} \mathbf{D}^{0}(\mathds{P}(V), A).
	\end{align*}
	The second of these is equivariant for the natural actions of $G(K) = \mathrm{GL}(r,K)$ on both sides, while the first isomorphism is equivariant for the actions of the stabilizer $G(O)Z(K)$ of $v_{0} \in G(K)$.
\end{Theorem}

As a direct consequence of the first isomorphism, i.e., of \eqref{Eq.Canonical-isomorphism-between-H(BT,A)-and-H(TV0,A)}, we find the following corollary, which is in keeping with the fact that bounded holomorphic functions on $\Omega$ are constant (see Proposition \ref{Proposition.Holomorphic-function-on-Omega-constant}).

\begin{Corollary}
	If $\varphi \in \mathbf{H}(\mathcal{BT}, A)$ has finite support, it vanishes identically.
\end{Corollary}

\begin{proof}
	Suppose that $\varphi$ has support in $\mathcal{BT}(n)$ with $n \in \mathds{N}$. Then its restriction to $\mathcal{T}_{v_{0}}(n+1)$ satisfies \eqref{Eq.Condition-1-H'n} and (4.6.5) at all vertices $v$ of 
	$\mathcal{T}_{v_{0}}(n+1)$. As $\mathcal{T}_{v_{0}}(n+1)$ is a finite tree, this forces $\varphi$ to vanish identically on $\mathcal{T}_{v_{0}}(n+1)$, thus on $\mathcal{BT}$.
\end{proof}

\section{Concluding remarks}

\subsection{} Ehud de Shalit in \cite{DeShalit01} Section 3.1 postulated four conditions $\mathfrak{A}, \mathfrak{B}, \mathfrak{C}, \mathfrak{D}$ for what he calls harmonic $k$-cochains on $\mathcal{BT}$. 
These conditions spezialized to $k=1$ are essentially our conditions (A), (B), (C) from \ref{subsection.Conditions-for-varphi-in-A(BT)}. Grosso modo, de Shalit's $\mathfrak{B}$ corresponds to (B), $\mathfrak{C}$ to 
(C) and $\mathfrak{D}$ to (A), while $\mathfrak{A}$ is a special case of (A).

\subsection{} In fact, the relationship with de Shalit's work is as follows. Suppose that $\mathrm{char}(K) = 0$, and consider the diagram
\begin{equation} \label{Eq.Relation-with-de-Shalits-work-and-this-article-for-char-K-0}
	\begin{tikzcd}
		u \ar[d, mapsto] 							& \mathcal{O}(\Omega)^{*} \ar[r, "P"]	\ar[d]																										& \mathbf{H}(\mathcal{BT}, \mathds{Z}) \ar[d, hook] \\
		d\log u = u^{-1}du				& \left\{ \begin{array}{l} \text{closed $1$-forms} \\ \text{on $\Omega$}\end{array} \right\} \ar[r, "\mathrm{res}"]	& \mathbf{H}(\mathcal{BT}, K) ~(= C_{\mathrm{har}}^{1} \text{ of \cite{DeShalit01}}),
	\end{tikzcd}
\end{equation}
where \enquote{$\mathrm{res}$} is de Shalit's residue mapping. Its commutativity follows for $u = \ell_{H,H'}$ from Corollary 7.6 and Theorem 8.2 of \cite{DeShalit01} (along with the explanations given there, and 
our description of $P(u)$), and may be verified for general $u$ by approximating. Hence the van der Put transform $P$ yields a concrete description of the residue mapping on logarithmic 1-forms. On the other hand, 
in characteristic $p$ the van der Put transform is finer than \enquote{$\mathop{d\log}$}, as the latter kills all $p$-powers.

\subsection{} Now suppose that $\mathrm{char}(K) = p > 0$, and that moreover $r=2$. Then $\mathcal{BT}$ is the Bruhat-Tits tree $\mathcal{T}$, and the residue mapping 
\[
	\mathrm{res} \colon \{ \text{1-forms on $\Omega = \Omega^{2}$}\} \longrightarrow \mathbf{H}(\mathcal{T}, C)
\]
(see \cite{Gekeler96} 1.8) is such that the diagram analogous with \eqref{Eq.Relation-with-de-Shalits-work-and-this-article-for-char-K-0}
\begin{equation}
	\begin{tikzcd}
		u \ar[d, mapsto]	& \mathcal{O}(\Omega)^{*} \ar[r, "P"]	\ar[d]							& \mathbf{H}(\mathcal{T}, \mathds{Z}) \ar[d]	\\
		d\log u	& \{\text{1-forms on $\Omega$}\} \ar[r, "\mathrm{res}"] 	&	\mathbf{H}(\mathcal{T}, C)
	\end{tikzcd}
\end{equation}
commutes, with remarkable arithmetic consequences (loc. cit., Sections 6 and 7). A similar residue map for $r>2$ unfortunately lacks so far. In any case, we should regard $P$ as a substitute for the logarithmic 
derivation operator
\[
	u \longmapsto d\log u = u^{-1} du
\] 
in characteristic 0.

\subsection{} In \cite{SchneiderStuhler91}, Peter Schneider and Ulrich Stuhler described the cohomology $H^{*}(\Omega, A)$ of $\Omega = \Omega^{r}$ with respect to an abstract cohomology theory, where 
$A = H^{0}(\mathrm{Sp}(K))$. That theory is required to satisfy four natural axioms, loc. cit., Section 2. As they explain, these axioms are fulfilled at least
\begin{itemize}
	\item for the étale $\ell$-adic cohomology of rigid-analytic spaces over $K$, where $\ell$ is a prime different from $p = \mathrm{char}(\mathds{F})$, and $A = \mathds{Z}_{\ell}$, and
	\item for the de Rham cohomology (where one must moreover assume that $\mathrm{char}(K) = 0$); here $A = K$.
\end{itemize}
Their result is stated loc. cit., Section 3, Theorem 1, which in dimension 1 is (in our notation)
\begin{equation}
	\begin{tikzcd}
		H^{1}(\Omega^{r}, A) \ar[r, "\cong"]	&\mathbf{D}^{0}(\mathds{P}(V^{\wedge}), A).
	\end{tikzcd}
\end{equation}
Theorem 8.2 in \cite{DeShalit01} gives that (in the case where $\mathrm{char}(K) = 0$ and $H^{*} = H_{\mathrm{dR}}^{*}$ is the de Rham cohomology)
\begin{equation}
	\begin{tikzcd}
		H_{\mathrm{dR}}^{k}(\Omega^{r}) \ar[r, "\cong"] & C_{\mathrm{har}}^{k},
	\end{tikzcd}
\end{equation}
where $C_{\mathrm{har}}^{1}$ is our $\mathbf{H}(\mathcal{BT}, K)$. Hence our Theorems \ref{Theorem.van-der-Put-sequence-short-exact-sequence-of-GK-modules} and 
\ref{Theorem.Abstraction-to-arbitrary-abelian-groups} refine the above in the case $k=1$. In \cite{AlonDeShalit02}, 
the authors relate the approaches of \cite{SchneiderStuhler91}, \cite{DeShalit01}, \cite{IovitaSpiess01} to the de Rham cohomology of $\Omega$. Specialized to $k=1$, this gives some more insight 
into our situation. In particular, it is possible to derive the surjectivity of the map $P$ in Theorem 3.11 also with the methods of \cite{AlonDeShalit02}, at least if $\mathrm{char}(K) = 0$.

\subsection{} Let now $\Gamma$ be a discrete subgroup of $G(K)$. The most interesting cases are those where the image of $\Gamma$ in $G(K)/Z(K) = \mathrm{PGL}(r,K)$ has finite covolume with respect
to Haar measure, or is even cocompact. Examples are given as Schottky groups in $\mathrm{PGL}(2,K)$ \cite{GerritzenVanDerPut80} or as arithmetic subgroups of $G(K)$ of different types, when $K$ is the 
completion $k_{\infty}$ of a global field $k$ at a non-archimedean place $\infty$ \cite{Drinfeld74}, \cite{Reiner75}. Then often the quotient analytic space $\Gamma \setminus \Omega$ is the set of $C$-points of an 
algebraic variety \cite{GoldmanIwahori63}, \cite{Drinfeld74}, \cite{Mustafin78}, which may be studied via a spectral sequence relating the cohomologies of $\Omega$ and $\Gamma$ with that of 
$\Gamma \setminus \Omega$ (\cite{SchneiderStuhler91} Section 5). For $r=2$, this essentially boils down to a study of the $\Gamma$-cohomology sequence of 
\eqref{Eq.van-der-Put-Short-exact-sequence-for-general-r} (\cite{Gekeler96} Section 5). But also for $r>2$, \eqref{Eq.van-der-Put-Short-exact-sequence-for-general-r} with its $\Gamma$-action will be useful, 
which is the topic of ongoing work.


\begin{bibdiv}
	\begin{biblist}
		\bib{AlonDeShalit02}{article}{title={On the cohomology of Drinfeld's $p$-adic symmetric domain}, author={Gil Alon and Ehud de Shalit}, journal={Israel Journal of Mathematics}, volume={129}, date={2002}, pages={1--20}}
		\bib{BruhatTits72}{article}{title={Groupes r\'eductifs sur un corps local.}, author={Fran\c cois Bruhat and Jacques Tits}, date={1972}, publisher={Institut des Hautes \'Etudes Scientifiques}, 
		journal={Publications Math\'ematiques de l'IH\'ES}, volume={42}, pages={5--251}}
		\bib{DeShalit01}{article}{title={Residues on buildings and de Rham cohomology of p-adic symmetric domains}, author={Ehud De Shalit}, date={2001}, volume={106}, publisher={Duke University Press}, journal={Duke Mathematical Journal}, pages={123--191}}
		\bib{DeligneHusemoller87}{article}{title={Survey of Drinfel'd modules}, author={Pierre Deligne and Dale Husemoller}, date={1987}, publisher={AMS}, journal={Contemporary Mathematics}, volume={67}, pages={25--91}}
		\bib{Drinfeld74}{article}{title={Elliptic modules (Russian), }, author={Vladimir Gershonovich Drinfel\cprime d}, date={1974}, journal={Mat. Sb. (N.S.)}, volume={94(136)}, pages={594--627.}}
		\bib{FresnelVanDerPut81}{book}{title={G{\'e}om{\'e}trie analytique rigide et applications}, author={Jean Fresnel and Marius van der Put}, date={1981}, publisher={Birkh{\"a}user, Basel}}
		\bib{FresnelVanDerPut04}{book}{author={Jean Fresnel and Marius van der Put}, title={Rigid analytic geometry and its applications}, publisher={Birkhäuser, Boston-Basel-Berlin}, date={2004}}		
		\bib{Gekeler96}{article}{title={Jacobians of Drinfeld modular curves}, author={Ernst{-}Ulrich Gekeler and Marc Reversat}, date={1996}, journal={Journal für die reine und angewandte Mathematik}, volume={476}, pages={27--93}}
		\bib{Gekeler17}{article}{title={On Drinfeld modular forms of higher rank}, author={Ernst{-}Ulrich Gekeler}, date={2017}, journal={Journal de Th\'eorie des Nombres de Bordeaux}, volume={29}, pages={875--902}}
		\bib{GekelerTA}{article}{title={On Drinfeld modular forms of higher rank II}, author={Ernst{-}Ulrich Gekeler}, journal={Journal of Number Theory, to appear}}
		\bib{GerritzenVanDerPut80}{book}{title={Schottky groups and Mumford curves}, author={Lothar Gerritzen and Marius van der Put}, date={1980}, publisher={Springer, Berlin}, series={Lecture Notes in Mathematics}}
		\bib{GoldmanIwahori63}{article}{title={The space of p-adic norms}, author={O. Goldman and N. Iwahori}, date={1963}, publisher={Institut Mittag-Leffler}, journal={Acta Mathematica}, volume={109}, pages={137--177}}
		\bib{IovitaSpiess01}{article}{title={Logarithmic differential forms on p-adic symmetric spaces}, author={Adrian Iovita and Michael Spiess}, date={2001}, publisher={Duke University Press}, journal={Duke Mathematical Journal}, volume={110}, pages={253--278}}
		\bib{Kiehl67}{article}{title={Theorem A und Theorem B in der nichtarchimedischen Funktionentheorie}, author={Reinhardt Kiehl}, date={1967}, journal={Inventiones Mathematicae}, volume={2}, pages={256--273}}
		\bib{Laumon96}{book}{title={Cohomology of Drinfeld Modular Varieties, Part 1, Geometry, Counting of Points and Local Harmonic Analysis}, author={Gérard Laumon}, date={1996}, publisher={Cambridge University Press},series={Cambridge Studies in Advanced Mathematics},}
		\bib{ManinDrinfeld73}{article}{title={Periods of p-adic Schottky groups}, author={Yuri Manin and Vladimir Gershonovich Drinfeld}, date={1973}, journal={Journal für die reine und angewandte Mathematik}, volume={0262\_0263}, pages={239--247}}
		\bib{Mustafin78}{article}{title={Non-Archimedean uniformization (Russian)}, author={G. A. Mustafin}, date={1978}, publisher={}, journal={Mat. Sb. (N.S.)}, volume={105 (147)}, pages={207--237}}
		\bib{Reiner75}{book}{title={Maximal orders}, author={Irving Reiner}, date={1975}, publisher={Academic Press, London-New York}, series={London Mathematical Society monographs, No. 5}}
		\bib{SchneiderStuhler91}{article}{title={The cohomology of p-adic symmetric spaces}, author={Peter Schneider and Ulrich Stuhler}, date={1991}, publisher={}, journal={Inventiones Mathematicae}, volume={105}, pages={47--122}}
		\bib{VanDerPut8182}{article}{title={Les fonctions thêta d'une courbe de Mumford}, author={Marius van der Put}, date={1981/82}, publisher={Secr\'etariat math\'ematique}, journal={Groupe de travail d'analyse ultram\'etrique}, volume={9}, pages={1-12, Institut Henri Poincaré, Paris, 1983}}
		
	\end{biblist}
\end{bibdiv}
\end{document}